\documentclass[12pt]{amsart}
\usepackage{amssymb,latexsym,amsmath,mathrsfs,amsthm,mathdots}
\usepackage{verbatim,graphics,graphicx,amscd}
\usepackage{enumitem,theoremref}
\usepackage{fullpage}

\newcommand{\R}{\mathbb{R}}
\newcommand{\C}{\mathbb{C}}
\newcommand{\Q}{\mathbb{Q}}
\newcommand{\Z}{\mathbb{Z}}

\newcommand{\der}{\mathop{\mathrm{der}}}

\newtheorem{thm}{Theorem}[section]

\newtheorem{cor}[thm]{Corollary}
\newtheorem{lemma}[thm]{Lemma}
\theoremstyle{remark}
\newtheorem*{remark}{Remark}

\theoremstyle{definition}

\newtheorem{defn}{Definition}[section]

\title{Angels' staircases, Sturmian sequences, \\
and trajectories on homothety surfaces}
\author{Joshua Bowman and Slade Sanderson}

\begin{document}

\begin{abstract}
A homothety surface can be assembled from polygons by 
identifying their edges in pairs via homotheties, which 
are compositions of translation and scaling. We consider 
linear trajectories on a $1$-parameter family of genus-$2$ 
homothety surfaces. The closure of a trajectory on each of 
these surfaces always has Hausdorff dimension $1$, and 
contains either a closed loop or a lamination with Cantor 
cross-section. Trajectories have cutting sequences that 
are either eventually periodic or eventually Sturmian. 
Although no two of these surfaces are affinely equivalent, 
their linear trajectories can be related directly to those 
on the square torus, and thence to each other, by means of 
explicit functions. We also briefly examine two related 
families of surfaces and show that the above behaviors can 
be mixed; for instance, the closure of a linear trajectory 
can contain both a closed loop and a lamination.
\end{abstract}

\maketitle

%\setcounter{section}{-1}
%\section*{Introduction}

A homothety of the plane is a similarity that preserves 
directions; in other words, it is a composition of 
translation and scaling. A homothety surface has an atlas 
(covering all but a finite set of points) whose transition 
maps are homotheties. Homothety surfaces are thus 
generalizations of translation surfaces, which have been 
actively studied for some time under a variety of guises 
(measured foliations, abelian differentials, unfolded 
polygonal billiard tables, etc.). Like a translation surface, 
a homothety surface is locally flat except at a finite set 
of singular points---although in general it does not have 
an accompanying Riemannian metric---and it has a well-defined 
global notion of direction, or slope, again except at the 
singular points. It therefore has, for each slope, a foliation 
by parallel leaves. Homothety surfaces admit affine deformations 
with globally-defined derivatives (up to scaling).

One can ask many of the same questions about homothety surfaces 
as are commonly asked about translation surfaces, for instance 
regarding the structure of their foliations and the affine 
automorphisms they admit. Homothety surfaces have appeared 
sporadically in the literature, but much remains to be learned 
about their dynamical properties.

In this article we study a one-parameter family of homothety 
surfaces $X_s$ in genus $2$ (see \S\ref{SS:mainexample} for 
their definition), focusing on the dynamical properties of 
their geodesics, which we refer to as linear trajectories. 
We show that the closure of any linear trajectory is nowhere 
dense, although in certain cases it is locally the product 
of a Cantor set and an interval. We also consider the cutting 
sequences of these trajectories and show that they are either 
eventually periodic or Sturmian sequences. As far as we know, 
this form of symbolic dynamics has not previously been 
approached for (non-translation) homothety surfaces. 
See \S\ref{SS:remarks} for further remarks on how our results 
relate to previous work.

The structure of the paper is as follows. 
In \S\ref{S:definitions} we provide background 
definitions, state our main results, and collect 
some well-known tools. 
Sections \ref{S:staircase} and \ref{S:symbolic} are 
largely technical, although they introduce some objects 
that may have broader interest. 
In \S\ref{S:proofs} we use the material of the 
preceding sections to prove our main results. 
Finally, in \S\ref{S:related} we show that, by modifying 
the construction of $X_s$, we can produce surfaces with 
linear trajectories that exhibit different dynamical 
behaviors in forward and backward time.

%\subsection*{Acknowledgements} The authors thank 
%Pepperdine University for support via the SURP 
%(Summer Undergraduate Research Program) and URF 
%(Undergraduate Research Fellowship) grants. 
%They are also grateful to Katherine Koch and Judy Wang 
%for early contributions to this project. We found 
%the notes \cite{Davis} by Diana Davis very helpful.

\section{Definitions and results}\label{S:definitions}

\subsection{Homothety surfaces}

\begin{defn}
A \emph{homothety} is a complex-affine map $h : \C \to \C$ 
of the form $h(z) = az + b$, where $a \in \R\setminus\{0\}$. 
If $a > 0$, then we say it is a \emph{direct} homothety.
\end{defn}

If $a = 1$ in the above definition, the homothety is a 
translation; otherwise, the homothety has a single fixed 
point in $\C$. In this paper, we will only work 
with direct homotheties. (Negative values of $a$ make it 
possible to generalize quadratic differentials, which also 
go by the name of ``half-translation surfaces'': maps of 
the form $h(z) = -z + b$ are half-translations in the sense 
that a composition of two such maps is a translation.) 
Homotheties are complex-analytic, and so they can be used 
to define complex structures on a surface. However, on a 
compact surface of genus greater than $1$, it is necessary 
to allow singularities at which the curvature of the surface 
is concentrated, and so we adopt the following definition.

\begin{defn}
A \emph{homothety surface} is a connected orientable surface 
$X$ together with a discrete subset $Z \subset X$ (called 
the \emph{singular set}) and an atlas on $X \setminus Z$ 
whose transition maps are homotheties, in such a way that 
the points of $Z$ become removable singularities with 
respect to the induced complex structure.
\end{defn}

We will suppress the dependence on the singular set $Z$ 
and the atlas in our notation and simply refer to the 
homothety surface $X$. The definition given in \cite{BFG17} 
is more restrictive than ours, in that it requires that 
each singular point have a neighborhood that is affinely 
equivalent to a Euclidean cone, but the difference is not 
relevant to the present work.

An important example is the quotient of an annulus 
$1 \le |z| \le r$ by the homothety $h(z) = rz$. The 
resulting surface has no singular points and is called 
a \emph{Hopf torus}.

A more general construction of homothety surfaces is to 
start with a finite collection of disjoint polygons 
$P_1,\dots,P_n$ in $\C$ and identify their edges in 
pairs via homotheties. The singular set is the image of 
the vertices of the polygons. (This construction is 
analogous to the well-known construction of translation 
surfaces from polygons.)

A homothety surface is automatically endowed with a flat 
connection on the tangent bundle of $X \setminus Z$, with 
holonomy in $\R^+ = (0,\infty)$. Unlike in the case of 
translation surfaces, this connection does not generally 
come from a Riemannian metric, and it does not provide a 
trivialization of the tangent bundle of $X \setminus Z$. 
It does, however, trivialize the circle bundle of directions, 
because scaling by a nonzero real number does not change the 
direction of a tangent vector. We therefore can refer 
to the \emph{slope} of any nonzero tangent vector based 
at a point of $X \setminus Z$; the slope takes values in 
$\mathbb{RP}^1 = \R \cup \{\infty\}$.

\subsection{Linear trajectories, cycles, and laminations}

The observations of the preceding paragraph make possible 
the following definitions.

\begin{defn}
Let $X$ be a homothety surface with singular set $Z$. 
A \emph{linear trajectory} on $X$ is a smooth curve 
$\tau : I \to X$, where $I \subseteq \R$ is an interval, 
such that the image of the interior of $I$ lies in 
$X \setminus Z$ and on this interior the tangent vector 
$\tau'$ has constant slope. 
A linear trajectory is \emph{critical} if $I$ includes 
at least one endpoint, and the image of this endpoint 
lies in $Z$. 
A linear trajectory is a \emph{saddle connection} if $I$ 
is compact and the image of both endpoints lies in $Z$.
\end{defn}

These definitions directly generalize the corresponding 
notions on translation surfaces. Note, however, that in 
the absence of a norm on the tangent bundle, it does not 
make sense to require that a linear trajectory have constant 
(much less unit) speed, though one can enforce that a 
trajectory have \emph{locally constant} speed, in the sense 
that $(z \circ \tau)'$ is constant in any local coordinate 
$z$. We adopt this assumption of locally constant speed, 
which corresponds to assuming linear trajectories are 
geodesics.

We will usually specify a trajectory $\tau$ by its 
starting point $\tau(0)$ and its slope $m$. These data 
do not completely determine the trajectory, as the 
direction of a trajectory can be reversed, so we will 
call the \emph{forward direction} of $\tau$ the 
parametrization for which the $x$-coordinate is locally 
increasing and the \emph{backward direction} the 
parametrization for which the $x$-coordinate is locally 
decreasing. (This assumes that $m \ne \infty$; for 
a vertical trajectory, we call the upward direction 
forward.) We also assume that each linear trajectory is 
\emph{maximal}, meaning that its domain cannot be extended 
to a larger interval.

\begin{defn}
A linear trajectory is \emph{closed} if its image 
is homeomorphic to a circle and it is not a saddle 
connection. The image of a closed trajectory is a 
\emph{(linear) cycle}.
\end{defn}

Periodic trajectories are closed, but a closed trajectory 
is not necessarily periodic, because it may return to the 
same point of $X$ with a different tangent vector (``at a 
different speed'').

\begin{defn}
A \emph{(linear) lamination} on a homothety surface $X$ is 
a nowhere-dense closed subset $\Lambda \subset X\setminus Z$ 
that is the union of the images of a collection of parallel 
linear trajectories, each of which is called a \emph{leaf} 
of $\Lambda$. A lamination $\Lambda$ is \emph{minimal} if 
it has a dense leaf.
\end{defn}

Note that we define a linear lamination $\Lambda$ to be 
a closed subset of $X \setminus Z$, not of $X$. This 
convention ensures that $\Lambda$ is a lamination in the 
usual topological sense: it is locally the product of an 
interval and another topological space (for example, a 
Cantor set). However, when we consider the closure 
$\overline\Lambda$ on $X$, this description may break down, 
if $\Lambda$ contains more than two critical trajectories 
with the same endpoint in $Z$. A lamination on $X$ carries 
a transverse affine structure, as in \cite{HatcherOertel}.

The simplest example of a lamination is a linear cycle; 
a single cycle is also an example of a minimal lamination. 
More generally, if the image of a linear trajectory is 
nowhere dense on $X$, then the closure of this image (in 
$X \setminus Z$) is a minimal lamination. A union of 
parallel cycles is an example of a non-minimal lamination.

On a translation surface, the closure of a trajectory is 
either a cycle, a saddle connection, or a subsurface. As 
we shall see, however, other kinds of minimal laminations 
can exist on a homothety surface.

\subsection{Affine maps and Veech group}\label{SS:affine}

The next definition carries over directly from the 
case of translation surfaces.

\begin{defn}
Let $X$ and $Y$ be homothety surfaces. A continuous, 
open map $\phi : X \to Y$ is called \emph{affine} if 
it is affine in local charts, excluding singularities. 
$X$ and $Y$ are \emph{affinely equivalent} if there 
exists an affine homeomorphism $\phi : X \to Y$. The group 
of affine self-maps of $X$ is written $\mathrm{Aff}(X)$.
\end{defn}

Each affine map $\phi$ has a derivative $\der\phi$ that 
is globally well-defined up to scaling (because scaling 
commutes with all other linear maps of $\R^2$). Hence we 
can normalize to assume that the derivative has determinant 
$\pm 1$, so that $\der\phi \in \mathrm{GL}(2,\R)/\R^+ \cong 
\mathrm{SL}(2,\R) \rtimes \{\pm1\}$.

\begin{defn}
Let $X$ be a homothety surface. The \emph{(generalized) 
Veech group} of $X$ is the image $\Gamma(X)$ of the 
derivative map 
$\der : \mathrm{Aff}(X) \to \mathrm{GL}(2,\R)/\R^+$.
\end{defn}

\begin{remark}
If we allow indirect homotheties in the construction of 
a homothety surface, then its Veech group is naturally 
a subgroup of $\mathrm{PGL}(2,\R)$ rather than 
$\mathrm{GL}(2,\R)/\R^+$.
\end{remark}

\subsection{Cylinders}

By a standard argument, every periodic trajectory produces 
a cycle that is contained in a Euclidean cylinder foliated 
by parallel, homotopic cycles. The same argument shows that 
a non-periodic closed trajectory produces a cycle that is 
contained in an affine cylinder foliated by non-parallel 
(but still homotopic) trajectories. In both cases, the 
boundary of the cylinder is formed of saddle connections.

In general, the maximal cylinder containing the image of 
a non-periodic closed trajectory is affinely equivalent to 
a subsurface of a Hopf torus cover, but we will not need 
this generality. Each affine cylinder we will consider 
can be obtained from a trapezoid in the plane by 
identifying its parallel sides via the unique homothety 
$h$ under which the longer side is mapped to the shorter. 
The non-identified sides of the trapezoid will then pass 
through the fixed point of $h$. The derivative $h'$ is the 
\emph{scaling factor} of the cylinder; by convention, 
$h' < 1$.

We remark that affine maps preserve cylinders and scaling 
factors.

\subsection{Attractors}

Another phenomenon that may occur on homothety surfaces but 
not translation surfaces is the existence of attractors.

\begin{defn}
Let $X$ be a homothety surface and $m \in \mathbb{RP}^1$ 
a slope. A \emph{forward attractor} for slope $m$ is a 
closed subset $\Sigma \subset X \setminus Z$ such that:
\begin{enumerate}
[label*={(\roman*)}]
\item every forward trajectory with slope $m$ that starts 
in $\Sigma$ remains in $\Sigma$ for all time;
\item there exists an open subset $U$ containing $\Sigma$ 
such that, for every open set $V$ containing $\Sigma$, any 
forward trajectory with slope $m$ that starts in $U$ is 
eventually always in $V$ (that is, if $\tau(0) \in U$, 
then there exists $t_0$ such that $\tau(t) \in V$ for all 
$t > t_0$);
\item $\Sigma$ is the closure of the image of a forward 
trajectory with slope $m$.
\end{enumerate}
The largest open set $U$ satisfying condition (ii) is 
called the \emph{basin of attraction} for $\Sigma$.
\end{defn}

A \emph{backward attractor} for slope $m$ is defined 
analogously. A backward attractor may also be called a 
\emph{forward repeller}, and vice versa. 

The simplest kind of attractor is an \emph{attracting 
cycle}, which is an attractor that is homeomorphic to 
$S^1$. As we saw in the previous section, an attracting 
cycle is contained in an affine cylinder with scaling 
factor $<1$; the interior of this cylinder is contained 
in the basin of attraction of the cycle. The scaling factor 
determines the ``rate of convergence'' to the attracting 
cycle of trajectories that enter the cylinder.

We shall also see examples of \emph{attracting laminations}. 
As before, our convention is that these are closed subsets 
of $X \setminus Z$, not of $X$; the reason is that otherwise 
attracting and repelling laminations could intersect at 
a singular point.

\subsection{Main example}\label{SS:mainexample}

We will primarily study a one-parameter family of homothety 
surfaces $X_s$, with $s \in (0,1)$. The construction of 
these surfaces is shown in Figure~\ref{F:Xs}.

\begin{figure}[h]
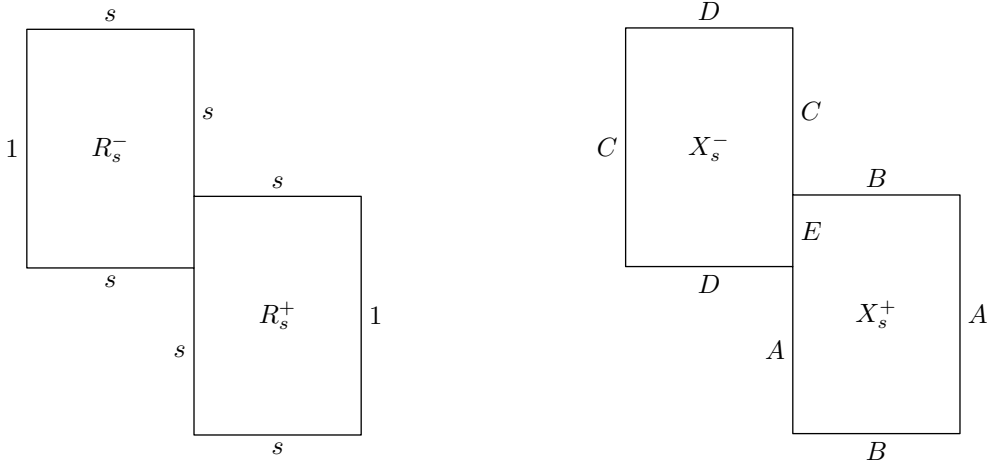

\begin{minipage}{2in}
\centering
\includegraphics{surfacefig-1.mps}
\end{minipage}
\hspace{1in}
\begin{minipage}{2in}
\centering
\includegraphics{surfacefig-2.mps}
\end{minipage}
\caption{{\sc Left:} The union of the rectangles $R_s^\pm$. 
{\sc Right:} Edge identifications that produce the 
surface $X_s$. Horizontal edges are identified by 
translations; vertical edges are identified by 
homotheties. $X_s^+$ and $X_s^-$ are genus $1$ subsurfaces 
with boundary. They are joined along the 
saddle connection $E$.}\label{F:Xs}
\end{figure}

To be explicit, we start with two rectangles $R_s^+$ and 
$R_s^-$ having horizontal side length $s$ and vertical 
side length $1$. The top of the left-hand side of $R_s^+$ 
is identified with the bottom of the right-hand side of 
$R_s^-$ along a segment $E$ of length $1 - s$ via an 
isometry, to create an eight-sided polygon. Each horizontal 
edge of $R_s^\pm$ is identified with the one directly above 
or below it via translation. Each vertical edge of length 
$1$ is identified with the opposite edge of length $s$ via 
homothety. The result is a genus $2$ surface $X_s$ with one 
singular point.

$X_s$ admits an involution $\rho_s : X_s \to X_s$ that exchanges 
$X_s^+$ and $X_s^-$ and may be visualized in Figure~\ref{F:Xs} 
as rotation around the midpoint of edge $E$. The derivative 
of $\rho_s$ is $-\mathrm{id}$. 

Another affine automorphism $\phi_s : X_s \to X_s$ is given 
by simultaneous Dehn twists in the vertical cylinders that 
pass through $X_s^+$ and $X_s^-$. The derivative of $\phi_s$ 
is $\left(\begin{smallmatrix}
1 & 0 \\ 1/s & 1
\end{smallmatrix}\right)$, so that a linear 
trajectory with slope $m \ne \infty$ is sent by 
$\phi_s$ to a trajectory of slope $m+1/s$.

 $\mathrm{Aff}(X_s)$ also contains an orientation-reversing 
 involution $\psi_s$ that has derivative 
 $\left(\begin{smallmatrix}
 1 & 0 \\ (1-s)/s & -1
 \end{smallmatrix}\right)$. This map $\psi_s$ reflects the 
 edges $A$, $C$, and $E$ across their respective midpoints. 
 It may be visualized as reflection across a horizontal axis 
 in a surface that is affinely equivalent to $X_s$.

\subsection{Cutting sequences}

Let $\tau : I \to X_s$ be a linear trajectory. We assume 
that the domain $I$ of $\tau$ is maximal, that $0 \in I$, 
and that if $\tau$ is a critical trajectory, then 
$\tau(0)$ is the singular point of $X_s$. If the image 
of $\tau$ is not entirely contained in any of the edges 
$A,B,C,D,E$, then $\tau$ meets these edges in sequence 
at a discrete set of times $0 < t_1 < t_2 < \cdots$. 
We define a word $c(\tau) = w_1 w_2, \dots$, called the 
\emph{cutting sequence} of $\tau$, by $w_\ell = L$ if 
$\tau(t_\ell) \in L$, where $L$ is taken from the 
alphabet $\{A,B,C,D,E\}$. The cutting sequence may be 
finite (if $\tau$ is a saddle connection) or infinite.

Let $X_s^+$ be the part of $X_s$ formed from $R_s^+$ 
and $X_s^-$ be the part of $X_s$ formed from $R_s^-$, 
as shown in Figure~\ref{F:Xs}. Then $X_s^+$ and $X_s^-$ 
are ``stable subsurfaces'' for forward and backward 
linear trajectories, respectively. That is, a trajectory that 
starts in $X_s^+$ will remain in $X_s^+$ in the forward 
direction, and a trajectory that starts in $X_s^-$ will 
remain in $X_s^-$ in the backward direction. Consequently, 
any linear trajectory on $X_s$ crosses the edge $E$, which 
joins $X_s^+$ and $X_s^-$, at most once. The cutting sequence 
of a typical trajectory is thus expected to consist of $C$s and 
$D$s in the negative direction and $A$s and $B$s in the 
positive direction, separated by one appearance of $E$.

\subsection{Main results}

The following two theorems about linear trajectories on 
the surfaces defined in \S\ref{SS:mainexample} summarize 
our main results.

\begin{thm}\label{T:main1}
Let $0 < s < 1$ be given, and let $X_s$ be the surface 
constructed as in Figure~\ref{F:Xs}.
\begin{enumerate}
[label*={(\roman*)}]
\item In the space of directions $\mathbb{RP}^1$, there 
is an open set $U_s$ of full measure such that, for all 
$m \in U_s$, $X_s$ has an attracting cycle $\Sigma^+$ and 
a repelling cycle $\Sigma^-$ in the direction $m$. The 
basin of attraction for either $\Sigma^+$ or $\Sigma^-$ 
is dense in $X_s$.
% If $m \neq \infty$, then all other trajectories 
% with slope $m$ are asymptotic to $\tau_m$. 
\item The complement of $U_s$ is a Cantor set $C_s$ whose 
Hausdorff dimension is $0$.
\item In $C_s$ there is a countable set $C'_s$ of directions 
$m$ that have a saddle connection $\tau_m$ such that all 
non-critical trajectories with slope $m$ are asymptotic to 
$\tau_m$.
\item For any direction $m \in C_s \setminus C'_s$, there is 
an attracting lamination $\Sigma^+$ and a repelling 
lamination $\Sigma^-$ in the direction $m$, each having 
a Cantor set cross-section with Hausdorff dimension $0$. 
The basin of attraction for $\Sigma^+$ is the complement 
of $\Sigma^-$.
\end{enumerate}
\end{thm}

Each connected component of the open set $U_s$ is associated 
with a rational number. The endpoints of these connected 
components form the set $C_s'$. Likewise, the points of 
$C_s \setminus C_s'$ are associated to irrational numbers. 
These associations are made more explicit by our second 
result, which characterizes the cutting sequences of 
trajectories on $X_s$ and relates them to cutting sequences 
on the square torus $T^2 = \R^2/\Z^2$. 
%Recall that a Sturmian sequence is the cutting 
%sequence of a trajectory on $T^2$ with irrational 
%slope. (See also \S\ref{S:symbolic}.)

\begin{thm}\label{T:main2}
Let $\tau$ be a forward trajectory on $X_s^+$ with slope 
$m$ and cutting sequence $c$.
\begin{enumerate}
[label*={(\roman*)}]
\item If $m \in U_s \cup C_s'$, then there exists 
$k/n \in \Q$ such that $c$ is eventually the same as 
the cutting sequence for a trajectory on $T^2$ having 
slope $k/n$.
\item If $m \in C_s \setminus C_s'$, then there exists 
$\xi \notin \Q$ such that $c$ is the same as the cutting 
sequence for a trajectory on $T^2$ having slope $\xi$.
\end{enumerate}
\end{thm}

\subsection{Remarks on results}\label{SS:remarks}

In the language of \cite{Liousse}, 
Theorem~\ref{T:main1}(i) says that the directions in $U_s$ 
(excluding those containing saddle connections) are 
``dynamically trivial''. The main result of \cite{Liousse} 
is that dynamically trivial foliations form an open dense 
subset, with respect to the $C^\infty$ topology, of 
oriented affine foliations having a fixed type of singular 
set. It is therefore not surprising that almost all 
directions on $X_s$ exhibit this behavior. 

When $s = 1/2$, $X_s$ is affinely equivalent to 
the ``two-chamber surface'' that appears in \cite{DFG16}. 
In that paper, the authors attempt to analyze the behavior 
of linear trajectories on the two-chamber surface, but they 
neglect the directions covered by Theorem~\ref{T:main1}(iv). 
Even though according to Theorem~\ref{T:main1}(ii) the set 
of these directions has Hausdorff dimension zero, their 
behavior is sufficiently interesting to merit thorough 
consideration, especially in light of the genericity result 
mentioned in the previous paragraph.

Several of the results of Theorem~\ref{T:main1} are similar 
to those obtained in \cite{BFG17} for another surface (the 
``disco surface''), but our methods are quite different. 
Principally, each of our surfaces has a relatively small 
Veech group (it is virtually cyclic, as correctly observed in 
\cite{DFG16} for the case $s = 1/2$), and so we cannot make 
extensive use of the theory of Fuchsian groups as is done 
in \cite{BFG17}. Instead, we relate linear trajectories on 
$X_s$ directly to those on the ordinary square torus $T^2$ 
by means of an ``angels' staircase'' function (see 
\S\ref{S:staircase}). We also make use of the theory 
of continued fractions (see \S\ref{S:symbolic}), which is 
akin to the use of Rauzy--Veech induction in \cite{BFG17}, 
but again our approach has a substantially different flavor.

Note that Theorem~\ref{T:main1} implies the following.

\begin{cor}\label{C:nodense}
No linear trajectory is dense in $X_s$ or in any 
subsurface of $X_s$.
\end{cor}

Corollary~\ref{C:nodense} contrasts with Conjecture~1 
of \cite{BFG17}, which states that on the disco surface 
some directions are minimal. This difference in behavior is 
likely due to the fact that $X_s$ has only one completely 
periodic direction ($m = \infty$), while the disco surface 
has many completely periodic directions because its 
Veech group is non-elementary.

In \S\ref{SS:AIET} we identify the piecewise-affine map 
$S^1 \to S^1$, associated to a direction $m \ne \infty$, 
which is induced by the first return of linear trajectories in 
the direction $m$ to a fixed vertical segment. The properties 
of this map provide the basis for several of our results. 
This map has been studied previously; see for example 
\cite{B93,BC99,Coutinho,Cetal,DH87,JansonOberg,LN18,V87}. 
Some of the results in \S\ref{S:staircase} reproduce parts of 
those earlier works. For the benefit of the reader, we provide 
full proofs of the properties we require, indicating overlaps 
where appropriate. A benefit of our approach is that the maps 
$S^1 \to S^1$ for various $m$ are realized simultaneously as 
sections of geodesic flow on a single surface $X_s$, thereby 
providing a unifying picture.

\subsection{Floor, ceiling, fractional part}

We use $\lfloor x \rfloor$ to denote the floor 
function, which returns the greatest integer not 
greater than $x$, and $\lceil x \rceil$ to denote 
the ceiling function, which returns the smallest 
integer not smaller than $x$. 
We also use $\{x\} = x - \lfloor x \rfloor$ to 
denote the fractional part of $x$. The function 
$x \mapsto \{x\}$ sends $\R$ to $[0,1)$ and satisfies 
the equation $e^{2\pi ix} = e^{2\pi i\{x\}}$, hence 
we will often treat $[0,1)$ as a coordinate on 
the circle $S^1$.

\subsection{Affine interval exchange transformations}
\label{SS:AIET}

We use \emph{affine interval exchange transformation} 
(AIET for short) to mean an injective piecewise-affine 
function $J \to J$, where $J \subset \R$ is a bounded 
interval (cf.\ \cite{MMY10}, where AIETs are assumed to 
be bijections). The prototypical example is a 
\emph{circle rotation} with parameter $\xi \in \R$, 
defined by $z \mapsto e^{2\pi i \xi} z$ on the circle 
$|z| = 1$ or by $x \mapsto \{x + \xi\}$ on the interval 
$[0,1)$. This is the first return map induced on a vertical 
segment of unit length in the square torus $T^2$ by the 
linear flow in the direction of slope $\xi$. 

We will primarily be interested in the AIET induced 
by the flow in one of the ``stable subsurfaces'' 
$X_s^\pm$ of $X_s$. Let $J^+ = J^- = [0,1)$; identify 
$J^+$ with the edge $A$ and $J^-$ with the edge $C$ in 
$X_s$ (cf.\ Figure~\ref{F:Xs}), each having coordinate 
$y$, running from bottom to top. When a slope $m \in \R$ 
is given, the forward linear flow in the direction $m$ 
on $X_s^+$ induces the AIET 
\[
y \mapsto \{s(y + m)\}
\]
on $J^+$. This is because the $y$-coordinate is first 
scaled by a factor of $s$ due to the identifications 
of the sides of $R_s^+$; then $y$ increases by $sm$ 
as a trajectory moves across the rectangle and the 
$x$-coordinate increases by $s$; then we take the 
fractional part of $sy + sm$ to return to the right 
side of the rectangle $R_s^+$. In \cite{Cetal,LN18}, 
this AIET is called a \emph{contracted rotation}.

In order to obtain from the linear flow on $X_s$ an 
AIET that is a bijection, we also need to consider 
the interval $J^-$. Unlike $J^+$, this interval 
is not invariant under the first-return map of the 
forward linear flow. The full first-return map on 
$J^+ \sqcup J^-$ of the linear flow in the direction 
$m$ is given by 
\[
y \mapsto 
\begin{cases}
\{s(y + m)\} \in J^+ 
& \text{if $y \in J^+$} \\
\{y + sm\} + \{s(1+m)\} \in J^+ 
& \text{if $y \in J^-$ and $\{y + sm\} < 1 - s$} \\
\{y/s+m\} \in J^-
& \text{if $y \in J^-$ and $1 - s \le \{y + sm\} < 1$}
\end{cases}
\]

\subsection{Infinite series}

Several of our results depend on known sums of infinite 
series. We will frequently make use of the familiar 
geometric series 
\begin{equation}\label{Eq:geometric}
\sum_{j=1}^\infty s^j = \frac{s}{1 - s}, 
\qquad |s| < 1.
\end{equation}
Closely related is the power series for the 
Koebe function 
\begin{equation}\label{Eq:koebe}
\sum_{j=1}^\infty js^j 
= s\frac{d}{ds} \frac{s}{1-s} 
= \frac{s}{(1 - s)^2}, 
\qquad |s| < 1.
\end{equation}
We will also need the sum of a certain series involving 
the Euler totient function $\varphi$:
\[
\varphi(n) = \#\{k : 1 \le k \le n,\; \gcd(k,n) = 1\}.
\]
The Lambert series for $\varphi$ is 
\begin{equation}\label{Eq:lambert}
\sum_{n=1}^\infty \frac{\varphi(n) s^n}{1 - s^n} 
= \frac{s}{(1 - s)^2}, 
\qquad |s| < 1.
\end{equation}
This formula was proved by Liouville \cite{Liouville}, 
and we present his proof here. Not only is it brief and 
elegant, it employs a technique of ``regrouping powers'' 
that we will often find useful. Start with Gauss's identity 
\[
j = \sum_{n\mid j} \varphi(n).
\]
Combining this with Equations \eqref{Eq:geometric} and 
\eqref{Eq:koebe}, we have 
\[
\frac{s}{(1 - s)^2} 
= \sum_{j=1}^\infty \bigg(\sum_{n\mid j} \varphi(n)\bigg) s^j 
= \sum_{j=1}^\infty \sum_{n\mid j} \big(\varphi(n) s^j\big)
= \sum_{n=1}^\infty \varphi(n) \sum_{\ell=1}^\infty s^{n\ell}
= \sum_{n=1}^\infty \varphi(n) \frac{s^n}{1 - s^n},
\]
where we have used the fact that $n\mid j$ if and only if 
$j = n\ell$ for some $\ell \ge 1$.

\subsection{Continued fractions}\label{Continued fractions}

Here we recall some basic facts about continued 
fractions. A reference is \cite{Khinchin}.

A \emph{finite continued fraction} is an expression of the 
form
\[
[a_0;a_1,a_2,\dots,a_n] = 
a_0 + 
\cfrac{1}
      {a_1 + \cfrac{1}
             {a_2 + \cfrac{1}
                    {\ddots + \cfrac{1}{a_n}}}}
\]
where $a_0 \in \Z$ and $a_1, \dots a_n \in \Z_+$. (In some 
sources, this is called a \emph{simple} continued fraction 
because every numerator is $1$; we will not be concerned 
with other types of continued fractions.) Every rational 
number can be written as a finite continued fraction, and 
the expression is unique provided $a_n \ne 1$. An 
\emph{infinite continued fraction} is a limit of finite 
continued fractions: 
\[
[a_0;a_1,a_2,a_3,\dots] 
= \lim_{n\to\infty} [a_0;a_1,\dots,a_n].
\]
Every irrational number can be written as an infinite 
continued fraction in a unique way.

In both the rational and irrational cases, the process of 
finding the continued fraction of a real number $x \in \R$ 
is the same. First set $x_0 = x$ and 
$a_0 = \lfloor x_0 \rfloor$, then compute $x_i$ and 
$a_i$ inductively: $x_{i+1} = 1/(x_i - a_i)$, $a_{i+1} = 
\lfloor x_{i+1} \rfloor$. If at some point $a_i = x_i$, 
then the process terminates; this occurs if and only if 
$x \in \Q$. Otherwise, the process continues forever. 
The terms of the sequence $a_i$ are called the 
\emph{partial quotients} of $x$. 

The finite continued fraction 
\[
\frac{P_i}{Q_i} = [a_0;a_1\dots,a_i]
\]
is called the \emph{$i$th convergent} of $x$. The 
numerators and denominators of the convergents can 
be computed recursively from the partial quotients 
as follows:
\begin{alignat}{5}
\label{Eq:Pconvrec}
P_{-1} &= 1 \hspace{0.5in}&
P_0 &= a_0 \hspace{0.5in}&
P_i &= a_i P_{i-1} + P_{i-2}, \\
\label{Eq:Qconvrec}
Q_{-1} &= 0 \hspace{0.5in}&
Q_0 &= 1 \hspace{0.5in}&
Q_i &= a_i Q_{i-1} + Q_{i-2}.
\end{alignat}
The convergents of $x$ alternate whether they are greater 
or less than $x$: 
\begin{equation}\label{Eq:convergentinequality}
\frac{P_{2i}}{Q_{2i}} \le x \le \frac{P_{2i+1}}{Q_{2i+1}}
\qquad\text{for all $i \ge 0$.}
\end{equation}
Of course, when $x \notin \Q$, both inequalities are always 
strict. We also have the inequality 
\begin{equation}\label{Eq:nearapproach}
|Q_i x - P_i| < \frac{1}{Q_{i+1}},
\end{equation}
and for any other rational number $P/Q$ with $Q \le Q_i$, 
we have $|Q_i x - P_i| < |Qx - P|$.

The \emph{intermediate fractions} of $x$ are rational 
numbers of the form 
\begin{equation}\label{Eq:intermediatefrac}
\frac{P_{i-2} + \alpha P_{i-1}}{Q_{i-2} + \alpha Q_{i-1}}, 
\qquad 0 \le \alpha \le a_i.
\end{equation}
In particular, convergents are intermediate fractions: 
when $\alpha = 0$ we get $P_{i-2}/Q_{i-2}$, and when 
$\alpha = a_i$ we get $P_i/Q_i$. The inequality 
\eqref{Eq:convergentinequality} implies that the sequence 
\eqref{Eq:intermediatefrac} is increasing with $\alpha$ 
when $i$ is even, decreasing when $i$ is odd.

Based on the preceding, we get information about certain 
integer multiples of $x$. Suppose $x > 0$. The inequalities 
\eqref{Eq:convergentinequality} and \eqref{Eq:nearapproach} 
imply that $\{Q_i x\}$ is close to $0$ when $i$ is even 
and close to $1$ when $i$ is odd. Moreover, since the 
intermediate fractions \eqref{Eq:intermediatefrac} lie 
between $P_{i-2}/Q_{i-2}$ and $P_i/Q_i$, the same is true 
for $\{(Q_{i-2} + \alpha Q_{i-1})x\}$ when 
$0 \le \alpha \le a_i$.  

We define a \emph{near approach to $0$} to be a remainder 
$\{jx\}$ such that $\{jx\}<\{j'x\}$ for all $j'<j$, and a 
\emph{near approach to $1$} to be a remainder $\{jx\}$ such 
that $\{jx\}>\{j'x\}$ for all $j'<j$.  We note that 
$j=Q_{i-2}+\alpha Q_{i-1}$ gives the near approaches $\{jx\}$ 
to $0$ when $i$ is even, and the near approaches to $1$ when 
$i$ is odd.  We may also conclude that, for all $i \ge 1$ and 
for all $0 \le \alpha \le a_i$,
\begin{equation}\label{int_frac_floor}
\lfloor (Q_{i-2} + \alpha Q_{i-1})x \rfloor = 
\begin{cases}
P_{i-2} + \alpha P_{i-1} & \text{when $i$ is even,} \\
P_{i-2} + \alpha P_{i-1} - 1 & \text{when $i$ is odd.}
\end{cases}\qquad
\end{equation}

The following lemma will be useful in Section~\ref{S:symbolic}.

\begin{lemma}\label{convergent_formulas}
Given $0<\xi<1$, let $[0;a_1,a_2,a_3,\dots]$ be its 
(finite or infinite) continued fraction, so that 
$\frac{1}{\xi+1}=[0;1,a_1,a_2,\dots]$ and 
$\frac{\xi}{\xi+1}=[0;a_1+1,a_2,a_3,\dots]$. 
If $\frac{P_i}{Q_i}$ is the $i^{\text{th}}$ convergent of 
$\xi$, $\frac{P'_i}{Q'_i}$ the $i^{\text{th}}$ convergent 
of $\frac{1}{\xi+1}$, and $\frac{P''_i}{Q''_i}$ the 
$i^{\text{th}}$ convergent of $\frac{\xi}{\xi+1}$, then 
\begin{align*}
& P'_i=Q_{i-1} 
&& P''_i=P_i\\
& Q'_i=P_{i-1}+Q_{i-1} 
&& Q''_i=P_i+Q_i
\end{align*}
for all $i\ge 0$.
\end{lemma}
\begin{proof}
By definition, $P'_{-1}=1$, and since $\frac{1}{\xi+1}<1$, 
$P'_0=0$.  Using the recursive formula \eqref{Eq:Pconvrec}, 
we see that $P'_1=1\cdot P'_0+P'_{-1}=1$.  Again, by definition, 
$Q_{-1}=0$ and $Q_0=1$.  We then have $P'_0=Q_{-1}$ and 
$P'_1=Q_0$.  Let $P'_1$ be the base case, and note that for 
$i\le 1$, $P'_i=Q_{i-1}$.  For the inductive step, suppose that 
$P'_{i-2}=Q_{i-3}$ and $P'_{i-1}=Q_{i-2}$.  Note that for all 
$i>1$, the $i^{\text{th}}$ position of the continued fraction 
of $\frac{1}{\xi+1}$ is $a_{i-1}$.  Equation~\eqref{Eq:Pconvrec} 
then tells us that $P'_{i}=a_{i-1}P'_{i-1}+P'_{i-2}$ for $i>1$. 
By our hypothesis and \eqref{Eq:Qconvrec}, this is equal to 
$Q_{i-1}=a_{i-1}Q_{i-2}+Q_{i-3}$.

Next, we aim to show that $Q'_i=P_{i-1}+Q_{i-1}$ for all 
$i\ge 0$.  By definition, $Q'_{-1}=0$ and $Q'_0=1$, and by 
\eqref{Eq:Qconvrec}, $Q'_1=1\cdot Q'_0+Q'_{-1}=1$.  Also by 
definition, $Q_{-1}=0$, $Q_0=1$, $P_{-1}=1$, and since 
$\frac{1}{\xi+1}<0$, $P_0=0$.  Let $Q'_1$ be the base case, 
and note that for all $i\le 1$, $Q'_i=P_{i-1}+Q_{i-1}$ because 
$Q'_{0}=P_{-1}+Q_{-1}=1$ and $Q'_1=P_{0}+Q_{0}=1$.  For the 
inductive step, suppose that $Q'_{i-2}=P_{i-3}+Q_{i-3}$ and 
$Q'_{i-1}=P_{i-2}+Q_{i-2}$.  As in the previous paragraph, we 
know that the $i^{\text{th}}$ position of the continued fraction 
of $\frac{1}{\xi+1}$ is $a_{i-1}$ for any $i>1$, so we have the 
recursive formula $Q'_{i}=a_{i-1}Q'_{i-1}+Q'_{i-2}$.  Using our 
hypothesis, this becomes
\[
a_{i-1}(P_{i-2}+Q_{i-2})+(P_{i-3}+Q_{i-3})
=(a_{i-1}P_{i-2}+P_{i-3})+(a_{i-1}Q_{i-2}+Q_{i-3})
=P_{i-1}+Q_{i-1}.
\]

A similar argument shows that $P''_i=P_i$ and $Q''_i=P_i+Q_i$. 
\end{proof}

\section{Angels' staircases}\label{S:staircase}

In this section we define two kinds of functions that 
will be essential to our study of linear trajectories 
on the surface $X_s$. One kind will be a ``parameter function'' 
that determines the type of behavior occurring in each 
direction on $X_s$. The other kind will be a ``dynamical 
function'' that determines a minimal invariant set in that 
direction. The connection between these two kinds of 
functions, given in Lemma~\ref{L:DU}, makes them useful for 
studying the ``contracting rotation'' that was defined in 
\S\ref{SS:AIET}.

As we shall see, each of these functions, of both kinds 
(with countably many exceptions among the dynamical functions), 
is strictly increasing and has a dense set of discontinuities. 
We call the graph of such a function an \emph{angels' 
staircase}. This name derives from the notion of an 
``inverted devil's staircase.'' 

%The inverse of an angels' staircase function can be extended 
%to a function on all of $\R$ whose graph is a devil's staircase 
%in the usual sense---that is, a continuous non-constant function 
%whose derivative exists and is equal to $0$ almost everywhere.

Given $s \in (0,1)$, we define the following two 
functions for all $x \in \R$:
\begin{align*}
\Delta_s^-(x) 
&= \left(\frac{1-s}{s}\right)^{\!2} 
   \sum_{j=1}^\infty s^j \lceil jx \rceil \\
\Delta_s^+(x) 
&= \frac{1-s}{s}
   + \left(\frac{1-s}{s}\right)^{\!2} 
     \sum_{j=1}^\infty s^j \lfloor jx \rfloor
\end{align*}
See Figure~\ref{F:staircase} for an example. 

\begin{figure}[h]
\includegraphics[scale=.275]{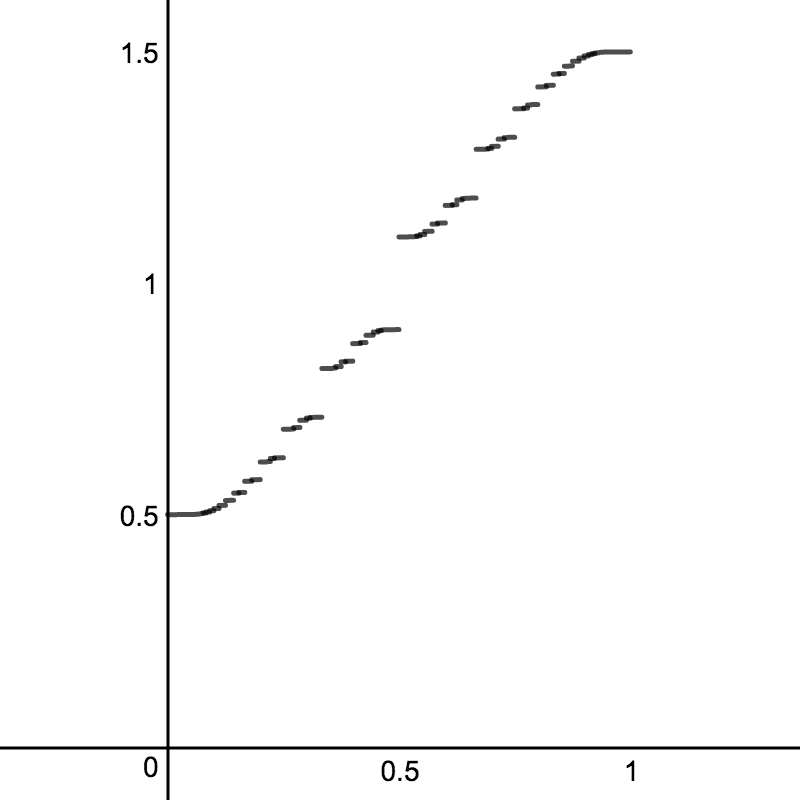}
\caption{The graph of $\Delta_s^-(x)$ when 
$s = 2/3$, drawn only for $0 \le x \le 1$. 
The graph of $\Delta_s^+(x)$ appears the same; 
the differences occur only at the jumps, which are 
dense but countable.}
\label{F:staircase}
\end{figure}

These functions are present implicitly in 
\cite{B93,BC99,Coutinho,DH87} and explicitly in 
\cite{JansonOberg}. The function $\Delta_s^+$ appears 
in \cite[\S6]{Cetal}, and in \cite{LN18} it is identified, 
with the roles of the parameter and the independent variable 
switched, as a ``Hecke--Mahler series'' (see also 
\cite[\S2]{KuipNeid}).

In Theorem~\ref{T:delta} we collect several useful 
properties of $\Delta_s^\pm$. Some parts of 
Theorem~\ref{T:delta} are direct generalizations of 
properties stated in \cite[Theorem~6.3]{BC02}, which 
assumes $s = 1/2$.

\begin{thm}\label{T:delta}
The functions $\Delta_s^-$ and $\Delta_s^+$ 
have the following properties:
\begin{enumerate}
[label*={(\roman*)}]
\item $\Delta_s^\pm$ is strictly increasing on $\R$.
\item For all $x \in \R$, $\Delta_s^\pm(x+1) = 
\Delta_s^\pm(x) + 1/s$.
\item $\Delta_s^-$ is left continuous: 
$\lim_{u\to x^-} \Delta_s^-(u) = \Delta_s^-(x)$ for all $x \in \R$.
\item $\Delta_s^+$ is right continuous: 
$\lim_{u\to x^+} \Delta_s^+(u) = \Delta_s^+(x)$ for all $x \in \R$.
\item If $x \notin \Q$, then $\Delta_s^-(x) = \Delta_s^+(x)$. 
\item If $x \notin \Q$ and $x > 0$, then 
\[
\Delta_s^-(x) = \Delta_s^+(x)
= \frac{1-s}{s} 
  \sum_{\ell=0}^\infty s^{\lfloor \ell/x \rfloor}.
\]
\item If $k$ and $n$ are integers such that $n > 0$ and 
$\gcd(k,n) = 1$, then 
\begin{align*}
\Delta_s^-(k/n) 
&= \frac{ks^{n-2}(1 - s)}{1 - s^n}
   + \frac{(1 - s)^2}{1 - s^n} 
     \sum_{r=1}^{n-1}s^{r-2}\left\lceil\frac{rk}{n}\right\rceil\\
\Delta_s^+(k/n) 
&= \frac{ks^{n-2}(1-s)}{1-s^n}
   + \frac{(1-s)^2}{1-s^n}
     \sum_{r=1}^{n-1} s^{r-2} \left\lfloor \frac{rk}{n} \right\rfloor 
   + \frac{1 - s}{s}
\end{align*}
and therefore $\Delta_s^+(k/n) - \Delta_s^-(k/n) = 
s^{n-2}(1-s)^2/(1-s^n)$. 
\item If $k$ and $n$ are positive integers 
such that $\gcd(k,n) = 1$, then 
\begin{align*}
\Delta_s^-(k/n)
&= \frac{1-s}{s}
   \sum_{\ell=0}^{\infty}s^{\left\lfloor\ell n/k\right\rfloor}\\
\Delta_s^+(k/n)
%&= \left(\frac{1-s}{s}\right)^{\!2}\frac{s^n}{1-s^n}
&= \frac{s^{n-2}(1-s)^2}{1-s^n}
   + \frac{1-s}{s}
   \sum_{\ell=0}^{\infty}s^{\left\lfloor\ell n/k\right\rfloor}
\end{align*}
\item If $k$ and $n$ are positive integers such that $k/n<1$ 
and $\gcd(k,n)=1$, then
\begin{align*}
\Delta_s^-(k/n)
&= \frac{(s-s^n)(1-s)}{s^2(1-s^n)}
   + \frac{s^n(1-s)}{s^2(1-s^n)}
   \sum_{\ell=1}^{k}(1/s)^{\left\lfloor (\ell-1)n/k\right\rfloor}\\
\Delta_s^+(k/n)
&= \frac{1-s}{s} 
   + \frac{s^n(1-s)}{s^2(1-s^n)}
   \sum_{\ell=1}^{k}(1/s)^{\left\lfloor (\ell-1)n/k\right\rfloor}.
\end{align*}
\item The set of discontinuities of $\Delta_s^\pm$ is 
precisely $\Q$.
\item The closure of the image of $\Delta_s^\pm$ is a 
Cantor set having Lebesgue measure $0$.
\item For all $x \in \R$, $\lim_{s\to1^-}\Delta_s^\pm(x) = x$.
\end{enumerate}
\end{thm}

\begin{proof}%[Proof of Theorem~\ref{T:delta}]
\hfill

\begin{enumerate}
[label*={(\roman*)}]
\item If $x < y$, then 
$\lceil jx \rceil \le \lceil jy \rceil$ 
and $\lfloor jx \rfloor \le \lfloor jy \rfloor$ for all 
$j \ge 1$. Moreover, for sufficiently large $j$ we have 
$1/j < y - x$, which means $jy > jx + 1$, in which case 
$\lceil jx \rceil < \lceil jy \rceil$ 
and $\lfloor jx \rfloor < \lfloor jy \rfloor$. These 
inequalities imply $\Delta_s^\pm(x) < \Delta_s^\pm(y)$, 
which is the desired result.

\item Substitute $x+1$ into the definition, expand, and use 
Equation~\eqref{Eq:koebe}:
\begin{align*}
\Delta_s^-(x+1) 
&= \left(\frac{1-s}{s}\right)^{\!2} 
   \sum_{j=1}^\infty s^j \lceil j(x+1) \rceil \\
&= \left(\frac{1-s}{s}\right)^{\!2} 
   \left(
   \sum_{j=1}^\infty s^j \lceil jx \rceil
   + \sum_{j=1}^\infty js^j 
   \right) \\
&= \left(\frac{1-s}{s}\right)^{\!2} 
   \sum_{j=1}^\infty s^j \lceil jx \rceil
   + \left(\frac{1-s}{s}\right)^{\!2} \frac{s}{(1-s)^2} \\
&= \Delta_s^-(x) + \frac1s.
\end{align*}
The calculation for $\Delta_s^+(x+1)$ is 
essentially identical.

\item The ceiling function is by definition left continuous, 
and therefore so is every term in the series that defines 
$\Delta_s^-(x)$. A sum of left continuous functions is left 
continuous, and so the partial sums of the series that 
defines $\Delta_s^-(x)$ are left continuous. Now it suffices 
to show that the series converges uniformly to $\Delta_s^-(x)$ 
on compact subsets of $\R$, because uniform convergence 
preserves left continuity. On $[-R,R]$, we have the 
inequality $\big|s^j \lceil jx \rceil\big| \le s^j(jR + 1)$. 
The series $\sum_{j=1}^\infty s^j (jR+1)$ converges (for 
instance, by the ratio test), and so the desired result 
follows from the Weierstrass $M$-test.

\item The proof is the same as that of part (iii), 
\emph{mutatis mutandis}.

\item Because $x \notin \Q$, $jx$ is never an integer when 
$j \ge 1$. Hence $\lceil jx \rceil = 1 + \lfloor jx \rfloor$ 
for all $j$, and we have 
\begin{align*}
\Delta_s^-(x) 
%&= \left(\frac{1-s}{s}\right)^2 
%   \sum_{j=1}^\infty s^j \lceil jx \rceil
&= \left(\frac{1-s}{s}\right)^{\!2} 
   \sum_{j=1}^\infty s^j \big(1 + \lfloor jx \rfloor\big) \\
&= \left(\frac{1-s}{s}\right)^{\!2} 
   \sum_{j=1}^\infty s^j 
   + \left(\frac{1-s}{s}\right)^{\!2} 
   \sum_{j=1}^\infty s^j \lfloor jx \rfloor \\
&= \left(\frac{1-s}{s}\right)^{\!2} 
   \frac{s}{1 - s} 
   + \left(\frac{1-s}{s}\right)^{\!2} 
   \sum_{j=1}^\infty s^j \lfloor jx \rfloor 
 = \Delta_s^+(x),
\end{align*}
as desired.

\item First we observe that, if $j$ and $\ell$ are integers, 
the quantity $j - \lfloor \ell/x \rfloor$ is positive when 
$j > \ell/x$;  by our assumptions on $x$, this inequality is 
equivalent to $\ell \le \lfloor jx \rfloor$. 
Therefore, given a fixed integer $j > 0$, the equation 
$j = i + \lfloor \ell/x \rfloor$ has $1 + \lfloor jx \rfloor 
= \lceil jx \rceil$ solutions $(i,\ell)$ with $i \ge 1$, 
$\ell \ge 0$. Combining this observation with 
Equation~\eqref{Eq:geometric}, we find 
\[
\frac{s}{1 - s} \sum_{\ell=0}^\infty s^{\lfloor \ell/x \rfloor} 
= \sum_{i=1}^\infty s^i
  \sum_{\ell=0}^\infty s^{\lfloor \ell/x \rfloor} 
= \sum_{i=1}^\infty \sum_{\ell=0}^\infty 
  s^{i+\lfloor \ell/x \rfloor} 
= \sum_{j=1}^\infty s^j \lceil jx \rceil.
\]
Applying this identity to the definition of $\Delta_s^-$, 
we obtain 
\[
\Delta_s^-(x) 
= \left(\frac{1-s}{s}\right)^{\!2}
  \sum_{j=1}^\infty s^j \lceil jx \rceil 
= \frac{1-s}{s}
   \sum_{\ell=0}^\infty s^{\lfloor \ell/x \rfloor}.
\]
By part (v), this last expression also equals $\Delta_s^+(x)$.

\item Each $j \ge 0$ can be written uniquely in the form 
$\ell n + r$, with $0 \le r \le n-1$. Thus 
\begin{align*}
\sum_{j=1}^\infty s^j \left\lceil \frac{jk}{n} \right\rceil
 = \sum_{j=0}^\infty s^j \left\lceil \frac{jk}{n} \right\rceil
&= \sum_{r=0}^{n-1} \sum_{\ell=0}^\infty s^{\ell n + r} 
   \left\lceil (\ell n + r)\frac{k}{n} \right\rceil \\
&= \sum_{r=0}^{n-1} \sum_{\ell=0}^\infty s^{\ell n + r} 
   \left( 
   \ell k + \left\lceil \frac{rk}{n} \right\rceil 
   \right) \\
&= \sum_{r=0}^{n-1} s^r 
   \left(
   k \sum_{\ell=0}^\infty \ell s^{\ell n} 
   + \left\lceil \frac{rk}{n} \right\rceil 
     \sum_{\ell=0}^\infty s^{\ell n}
   \right) \\
&= \frac{1 - s^n}{1 - s} \cdot \frac{ks^n}{(1 - s^n)^2}
   + \frac{1}{1 - s^n} \sum_{r=0}^{n-1} s^r 
     \left\lceil \frac{rk}{n} \right\rceil
\end{align*}
and so 
\begin{align*}
\Delta_s^-(k/n) 
&= \left(\frac{1-s}{s}\right)^{\!2}
   \left(
   \frac{ks^n}{(1-s)(1-s^n)} 
   + \frac{1}{1 - s^n} \sum_{r=1}^{n-1} s^r 
     \left\lceil \frac{rk}{n} \right\rceil
   \right) \\
&= \frac{ks^{n-2}(1 - s)}{1 - s^n}
   + \frac{(1 - s)^2}{1 - s^n} 
     \sum_{r=1}^{n-1}s^{r-2}\left\lceil \frac{rk}{n} \right\rceil
\end{align*}
as claimed. 
%(This proof follows a method outlined by Kuipers--Neiderreiter.)
The formula for $\Delta_s^+(k/n)$ is obtained analogously.

Next we observe that, because $k/n$ is in reduced form, $rk/n$ 
is not an integer when $1 \le r \le n-1$, and for such values 
of $r$ we have $\lceil rk/n \rceil=\lfloor rk/n \rfloor +1$. 
Thus we find 
\begin{align*}
\Delta_s^+(k/n)-\Delta_s^-(k/n)&=\frac{ks^{n-2}(1-s)}{1-s^n}
   + \frac{(1-s)^2}{1-s^n}
     \sum_{r=1}^{n-1} s^{r-2} \left\lfloor \frac{rk}{n} \right\rfloor 
   + \frac{1 - s}{s}\\
   &\qquad-\frac{ks^{n-2}(1 - s)}{1 - s^n}
   - \frac{(1 - s)^2}{1 - s^n} 
     \sum_{r=1}^{n-1}s^{r-2}\left\lceil \frac{rk}{n} \right\rceil\\
   &=\frac{(1-s)^2}{1-s^n}
     \sum_{r=1}^{n-1} s^{r-2} \left\lfloor \frac{rk}{n} \right\rfloor 
   + \frac{1 - s}{s}\\
   &\qquad-\frac{(1 - s)^2}{1 - s^n} 
     \sum_{r=1}^{n-1}s^{r-2}\bigg(\left\lfloor \frac{rk}{n} \right\rfloor +1\bigg)\\
   &=\frac{1-s}{s}-\frac{(1-s)^2}{1-s^n}\sum_{r=1}^{n-1}s^{r-2}
    =\frac{1-s}{s}-\frac{(1-s)^2}{1-s^n}\cdot\frac{1-s^{n-1}}{s(1-s)}\\
   &=\frac{(1-s)(1-s^n)-(1-s)(1-s^{n-1})}{s(1-s^n)}
    =\frac{s^{n-2}(1-s)^2}{1-s^n}.
\end{align*}
\item Observe that the function 
$\frac{1-s}{s} \sum_{\ell=0}^\infty s^{\lfloor \ell/x \rfloor}$ 
is left continuous, and so the formula for $\Delta_s^-(k/n)$ 
follows from the irrational case (part (vi)) and the fact that 
$\Delta_s^-$ is also left continuous (part (iii)). The formula 
for $\Delta_s^+(k/n)$ then follows from part (vii).

\item 
Notice that
\begin{align*}
\sum_{\ell=1}^{k}(1/s)^{\left\lfloor (\ell-1)n/k\right\rfloor}
&=1+\sum_{\ell=1}^{k-1}(1/s)^{\left\lfloor \ell n/k\right\rfloor}
 =1+\sum_{\ell=1}^{k-1}(1/s)^{\left\lfloor (k-\ell)n/k\right\rfloor}\\
&=1+\sum_{\ell=1}^{k-1}(1/s)^{n-\left\lceil \ell n/k\right\rceil}
 =1+\sum_{\ell=1}^{k-1}(1/s)^{n-\left\lfloor \ell n/k\right\rfloor -1}\\
&=1+\frac{1}{s^{n-1}}\sum_{\ell=1}^{k-1}s^{\left\lfloor \ell n/k\right\rfloor}.
\end{align*}
Our proposed $\Delta_s^-(k/n)$ then becomes
\begin{align*}
\frac{(s-s^n)(1-s)}{s^2(1-s^n)}+\frac{s^n(1-s)}{s^2(1-s^n)}\left(1+\frac{1}{s^{n-1}}\sum_{\ell=1}^{k-1}s^{\left\lfloor \ell n/k\right\rfloor}\right)&=\frac{1-s}{s(1-s^n)}+\frac{1-s}{s(1-s^n)}\sum_{\ell=1}^{k-1}s^{\left\lfloor \ell n/k\right\rfloor}\\
&=\frac{1-s}{s(1-s^n)}\sum_{\ell=0}^{k-1}s^{\left\lfloor \ell n/k\right\rfloor}\\
&=\frac{1-s}{s}\sum_{j=0}^{\infty}s^{jn}\sum_{\ell=0}^{k-1}s^{\left\lfloor \ell n/k\right\rfloor}\\
&=\frac{1-s}{s}\sum_{\ell=0}^{\infty}s^{\left\lfloor\ell n/k \right\rfloor},
\end{align*}
which, as we saw in part (viii), is $\Delta_s^-(k/n)$. 
Again, from part (vii), we know that 
$\Delta_s^+(k/n)=\frac{s^{n-2}(1-s)^2}{1-s^n}+\Delta_s^-(k/n)$. 
Then 
\begin{align*}
\Delta_s^+(k/n)
&= \frac{s^{n-2}(1-s)^2}{1-s^n}
   + \frac{(s-s^n)(1-s)}{s^2(1-s^n)}
   + \frac{s^n(1-s)}{s^2(1-s^n)}
     \sum_{\ell=1}^{k}
     (1/s)^{\left\lfloor (\ell-1)n/k\right\rfloor}\\
&= \frac{\big(s^n(1-s)+s-s^n\big)(1-s)}{s^2(1-s^n)}
   + \frac{s^n(1-s)}{s^2(1-s^n)}
     \sum_{\ell=1}^{k}
     (1/s)^{\left\lfloor (\ell-1)n/k\right\rfloor}\\
&= \frac{(s-s^{n+1})(1-s)}{s^2(1-s^n)}
   + \frac{s^n(1-s)}{s^2(1-s^n)}
     \sum_{\ell=1}^{k}
     (1/s)^{\left\lfloor (\ell-1)n/k\right\rfloor}\\
&= \frac{1-s}{s}
   + \frac{s^n(1-s)}{s^2(1-s^n)}
     \sum_{\ell=1}^{k}
     (1/s)^{\left\lfloor (\ell-1)n/k\right\rfloor}
\end{align*}
as proposed.

\item By part (i), $\Delta_s^-$ and $\Delta_s^+$ are monotone 
functions, and so they have one-sided limits at every point. 
By part (v), they are equal on the set of irrationals, which 
is dense in $\R$, and so at every point they have the same 
one-sided limits as each other. Thus, at an irrational number 
their one-sided limits match by parts (iii) and (iv), and at 
a rational number their one-sided limits are different 
by part (vii).

\item First we show that the closure of the image of 
$\Delta_s^\pm$ has measure zero. By the translation 
property in part (ii), it suffices to show that the 
closure of the image of $\Delta_s^\pm$ in $[0,1/s)$ 
has measure zero. 

Observe that the complement of the image of either 
$\Delta_s^-$ or $\Delta_s^+$ contains the open interval 
$(\Delta_s^-(k/n),\Delta_s^+(k/n))$ for any rational number 
$k/n$. If $k/n$ is reduced, then by part (vii) the length 
of this open interval depends only on $n$, not on $k$.

Let $\varphi$ be the Euler totient function. Then, for fixed 
$n \ge 1$, there are $\varphi(n)$ intervals in $[0,1/s)$ of 
the form $(\Delta_s^-(k/n),\Delta_s^+(k/n))$, each having 
length $s^{n-2}(1-s)^2/(1-s^n)$. Summing over all $n$, we 
find that the total length of these open intervals is 
\[
\sum_{n=1}^\infty \varphi(n) \frac{s^{n-2}(1-s)^2}{1-s^n} 
= \frac{(1-s)^2}{s^2} 
  \sum_{n=1}^\infty \frac{\varphi(n) s^n}{1 - s^n}
= \frac{(1-s)^2}{s^2} \cdot \frac{s}{(1 - s)^2} 
= \frac1s,
\]
where we have used Equation~\eqref{Eq:lambert} to obtain the 
second equality. Therefore the complement of the image of 
$\Delta_s^\pm$ contains an open set of full measure in 
$[0,1/s)$, which means that the closure of the image of 
$\Delta_s^\pm$ has measure zero.

To show that the closure of the image of $\Delta_s^\pm$ 
is a Cantor set, by Brouwer's characterization it suffices 
to show that it is perfect and totally disconnected. (We 
already know that it is compact and metrizable, since 
it is a closed subset of $\mathbb{RP}^1$.) The fact that 
it is perfect follows from either the left-continuity of 
$\Delta_s^-$ or the right-continuity of $\Delta_s^+$. 
The fact that it is totally disconnected follows from 
its measure being zero.

\item We first prove the result for $x \in \Q$, using 
the formulas from part (vii). Note that 
\[
\frac{1-s}{1-s^n} = \frac{1}{1+s+\cdots+s^{n-1}},
\]
which tends to $1/n$ as $s \to 1$. The first term in 
the expression for either $\Delta_s^-(k/n)$ or 
$\Delta_s^+(k/n)$ given in part (vii) thus approaches 
$k/n$ as $s \to 1$, while all other terms approach $0$, 
due to an additional factor of $1 - s$.

The result for $x \notin \Q$ now follows from the fact 
that $\Delta_s^-$ and $\Delta_s^+$ are monotone. That is, 
given any $x \notin \Q$ and any $\varepsilon > 0$, let 
$r_1$ and $r_2$ be rational numbers such that 
\[
x - \frac{\varepsilon}{2} < r_1 < 
x < r_2 < x + \frac{\varepsilon}{2},
\]
and choose $s$ such that 
$|\Delta_s^\pm(r_1) - r_1| < \varepsilon/2$ and 
$|\Delta_s^\pm(r_2) - r_2| < \varepsilon/2$. 
By part (i), we have 
\[
\Delta_s^\pm(r_1) < \Delta_s^\pm(x) < \Delta_s^\pm(r_2)
\]
which implies 
\[
r_1 - \frac{\varepsilon}{2} < \Delta_s^\pm(x) < 
r_2 + \frac{\varepsilon}{2}
\]
and consequently $x - \varepsilon < \Delta_s^\pm(x) < 
x + \varepsilon$, or $|\Delta_s^\pm(x) - x| < \varepsilon$.
\qedhere
\end{enumerate}
\end{proof}

When we examine how the gaps in the image of $\Delta_s^\pm$ 
vary with $s$, we see an ``Arnold tongues''-type phenomenon, 
illustrated in Figure~\ref{F:tongues}. The curves that bound 
each tongue in this figure are algebraic; 
Theorem~\ref{T:delta}(vii) provides explicit formulas for 
them. (The curves corresponding to irrational values of $\xi$ 
are transcendental, however.) Theorem~\ref{T:delta}(xi) says 
that the intersection of this figure with a vertical segment 
always has measure $0$.

\begin{figure}[h]
\includegraphics[scale=.45]{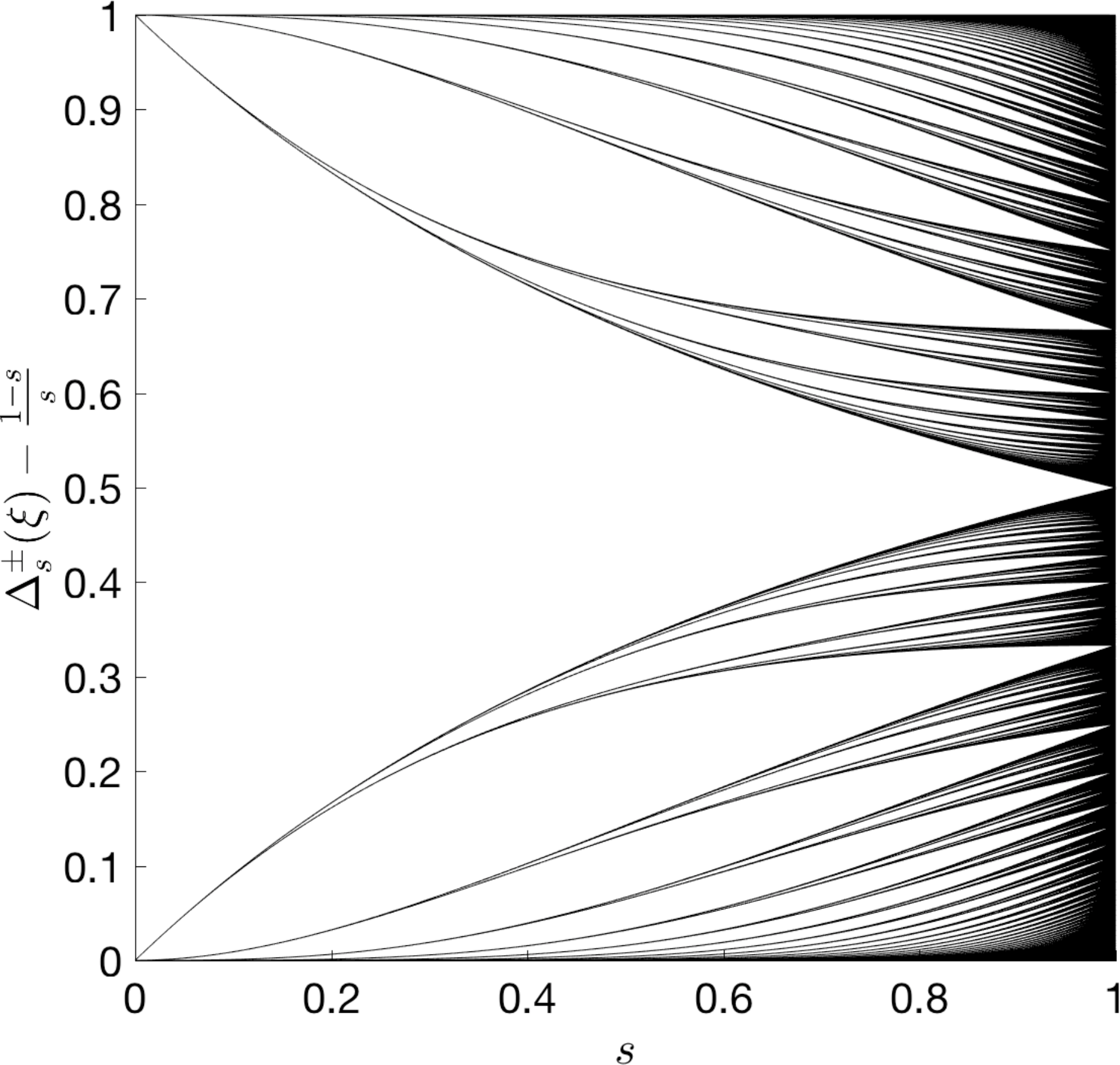}
\caption{``Tongues of angels'': Each tongue 
corresponds to a rational number in $(0,1)$. 
The boundary curves are drawn using the formulas for 
$\Delta_s^\pm(k/n)$.}\label{F:tongues}
\end{figure}

Next, given $s \in (0,1)$ and $\xi \in (0,\infty)$, we 
define two functions $\Upsilon_{s,\xi}^\pm : \R \to \R$ by 
\[
\Upsilon_{s,\xi}^-(x) 
= \frac{1-s}{s}
  \sum_{j=1}^\infty s^j \lceil x - \{j\xi\} \rceil, \qquad
\Upsilon_{s,\xi}^+(x) 
= 1 + \frac{1-s}{s}
  \sum_{j=1}^\infty s^j \lfloor x - \{j\xi\} \rfloor.
\]
(Recall that $\{x\}$ is the fractional part function.) 
A function similar to $\Upsilon_{s,\xi}^+$ is introduced 
in \cite[\S{II.2.1}]{Coutinho} and appears also in 
\cite{JansonOberg}. See Figure~\ref{F:upsilon} for examples. 
Theorem~\ref{T:upsilon} collects several properties of 
$\Upsilon_{s,\xi}^\pm$.

\begin{figure}[h]
\includegraphics[scale=.25]{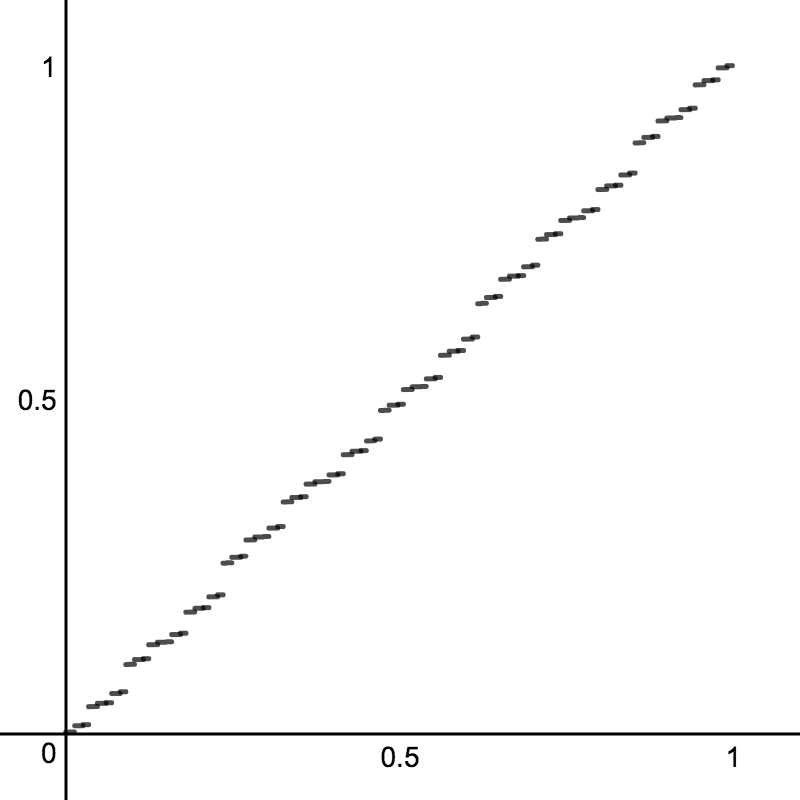}
\hspace{0.75in}
\includegraphics[scale=.25]{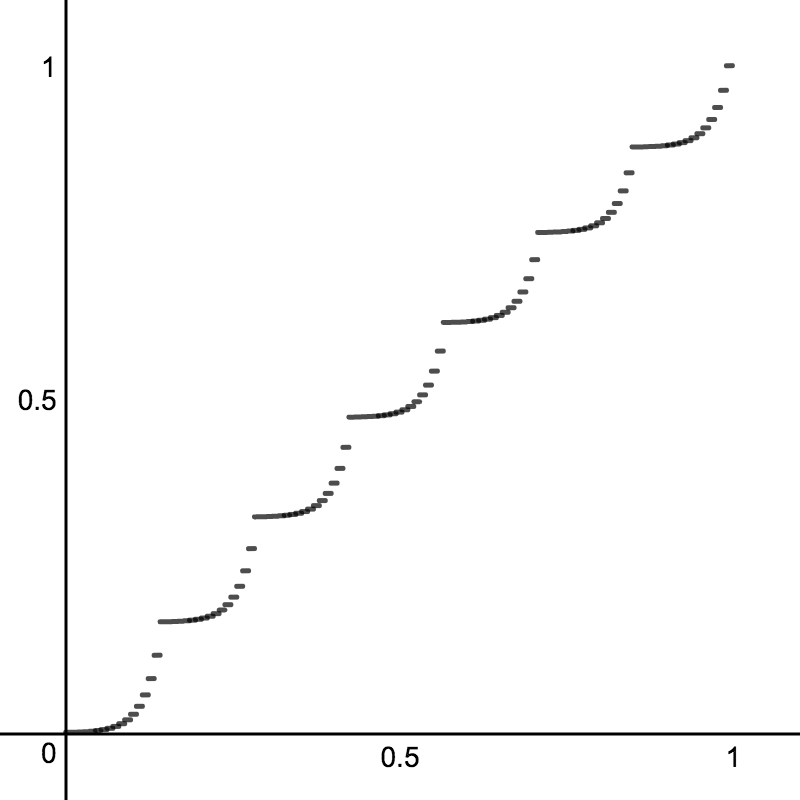}
\caption{The graphs of two functions $\Upsilon_{s,\xi}^\pm(x)$ 
with $s = 0.95$, drawn only for $0 \le x \le 1$. 
{\sc Left:} $\xi = (1+\sqrt5)/2$.
{\sc Right:} $\xi = \pi$.}\label{F:upsilon}
\end{figure}

\begin{thm}\label{T:upsilon}
The functions $\Upsilon_{s,\xi}^\pm$ have the following 
properties:
\begin{enumerate}
[label*={(\roman*)}]
\item $\Upsilon_{s,\xi}^\pm = \Upsilon_{s,\{\xi\}}^\pm$.
\item For all $x \in \R$, $\Upsilon_{s,\xi}^\pm(x+1) = 
\Upsilon_{s,\xi}^\pm(x)+1$.
\item $\Upsilon_{s,\xi}^-$ is left continuous.
$\Upsilon_{s,\xi}^+$ is right continuous.
\item $\Upsilon_{s,\xi}^\pm$ is monotone non-decreasing.
If $\xi \notin \Q$, then $\Upsilon_{s,\xi}^\pm$ is 
strictly increasing.
\item Suppose $\xi \notin \Q$. If $x = a + b\xi$ for 
some $a \in \Z$ and $b \in \Z^+$, then 
$\Upsilon_{s,\xi}^+(x) - \Upsilon_{s,\xi}^-(x) = 
(1-s)s^{b-1}$. For all other values of $x$, 
$\Upsilon_{s,\xi}^+(x) = \Upsilon_{s,\xi}^-(x)$.
\item If $k$ and $n$ are positive integers such that 
$\gcd(k,n) = 1$, then for all $x \in (0,1]$
\[
\Upsilon_{s,k/n}^-(x) 
= \frac{1 - s}{s}
  \left(
  -1 + \frac{1}{1 - s^n}
  \sum_{r=0}^{n-1} s^r 
  \left\lceil x - \left\{\frac{rk}{n}\right\}\right\rceil
  \right),
\]
and for all $x \in [0,1)$
\[
\Upsilon_{s,k/n}^+(x) 
= 1 + \frac{1 - s}{s}\cdot\frac{1}{1 - s^n}
  \sum_{r=1}^{n-1} s^r 
  \left\lfloor x - \left\{\frac{rk}{n}\right\}\right\rfloor.
\]
\item If $\xi\notin\Q$, then the closure of the image of 
$\Upsilon_{s,\xi}^\pm$ is a Cantor set having Lebesgue 
measure $0$.
\item If $\xi\notin\Q$, then for all $x \in \R$, 
$\lim_{\upsilon\to\xi}\Upsilon_{s,\upsilon}^\pm(x) 
= \Upsilon_{s,\xi}^\pm(x)$.
\item If $k$ and $n$ are positive integers such that 
$\gcd(k,n) = 1$, then for all $x \in \R$, 
\[
\lim_{s\to1^-}\Upsilon_{s,k/n}^-(x) = 
\frac1n\lceil nx \rceil 
\qquad\text{and}\qquad
\lim_{s\to1^-}\Upsilon_{s,k/n}^+(x) = 
\frac1n\lfloor nx+1 \rfloor.
\]
\item If $\xi\notin\Q$, then for all $x \in \R$, 
$\lim_{s\to1^-}\Upsilon_{s,\xi}^\pm(x) = x$. 
\end{enumerate}
\end{thm}

\begin{proof}\hfill

\begin{enumerate}
[label*={(\roman*)}]
\item Clear from the equality 
$\{j\xi\} 
= \big\{ j (\lfloor\xi\rfloor + \{\xi\}) \big\}
= \big\{ j\lfloor\xi\rfloor + j\{\xi\} \big\}
= \big\{ j \{\xi\} \big\}$.
\item Substitute $x+1$ into the definition of 
$\Upsilon_{s,\xi}^-$ and expand:
\begin{align*}
\Upsilon_{s,\xi}^-(x+1) 
&= \frac{1-s}{s}
   \sum_{j=1}^\infty s^j \lceil x + 1 - \{j\xi\} \rceil 
 = \frac{1-s}{s}
   \sum_{j=1}^\infty 
   s^j \big(\lceil x - \{j\xi\} \rceil + 1\big) \\
&= \frac{1-s}{s}
   \sum_{j=1}^\infty 
   s^j \lceil x - \{j\xi\} \rceil 
   + \frac{1 - s}{s} \sum_{j=1}^\infty s^j 
 = \Upsilon_{s,\xi}^-(x) + 1.
\end{align*}
The proof for $\Upsilon_{s,\xi}^+$ is essentially 
identical.

\item Same reasoning as in the proof of 
Theorem~\ref{T:delta}(iii).

\item If $x < y$, then 
$\lceil x - \{j\xi\}\rceil \le \lceil y - \{j\xi\} \rceil$ 
and $\lfloor x - \{j\xi\}\rfloor \le 
\lfloor y - \{j\xi\} \rfloor$ for all $j \ge 1$, which 
implies that $\Upsilon_{s,\xi}^\pm$ is non-decreasing, 
because the terms in the series that define 
$\Upsilon_{s,\xi}^\pm(y)$ are at least as great as 
those in the series for $\Upsilon_{s,\xi}^\pm(x)$.

Suppose $\xi\notin\Q$ and $0 \le x < y < 1$. The set of 
numbers of the form $\{j\xi\}$ is dense in $[0,1)$, and 
so there exists some $j$ such that $x < \{j\xi\} < y$; 
for such $j$ we have $\lceil x - \{j\xi\} \rceil = 0$ 
and $\lceil y - \{j\xi\} \rceil = 1$. This demonstrates 
that at least one term in the series defining 
$\Upsilon_{s,\xi}(y)$ is strictly greater than the 
corresponding term in the series for $\Upsilon_{s,\xi}(x)$, 
and so $\Upsilon_{s,\xi}(x) < \Upsilon_{s,\xi}(y)$. 
The general result that $x < y$ implies 
$\Upsilon_{s,\xi}(x) < \Upsilon_{s,\xi}(y)$ when 
$\xi\notin\Q$ follows from this special case and part (ii).

\item For points in $[0,1)$, this follows from the fact 
that a jump of size $(1-s)s^{b-1}$ occurs at $\{b\xi\}$. 
The general result then follows from part (ii).

\item 
Each $j \ge 0$ can be written uniquely in the form 
$\ell n + r$, with $0 \le r \le n-1$. Using the fact that 
$\lceil a \rceil = 1$ for $a \in (0,1]$, we find that 
\begin{align*}
\sum_{j=1}^\infty s^j 
\left\lceil x - \left\{\frac{jk}{n}\right\} \right\rceil 
&= -1 + 
   \sum_{j=0}^\infty s^j 
   \left\lceil 
     x - \left\{\frac{jk}{n}\right\} 
   \right\rceil \\
&= -1 + 
   \sum_{\ell=0}^\infty \sum_{r=0}^{n-1} s^{\ell n + r} 
   \left\lceil 
     x - \left\{\frac{(\ell n + r)k}{n}\right\} 
   \right\rceil \\
&= -1 + 
   \sum_{\ell=0}^\infty s^{\ell n} 
   \sum_{r=0}^{n-1} s^r 
   \left\lceil 
     x - \left\{\ell k + \frac{rk}{n}\right\} 
   \right\rceil \\
&= -1 + 
   \frac{1}{1 - s^n} 
   \sum_{r=0}^{n-1} s^r
   \left\lceil 
     x - \left\{\frac{rk}{n}\right\} 
   \right\rceil.
\end{align*}
The stated result for $\Upsilon_{s,k/n}^-$ follows from 
this equality. 

The proof for $\Upsilon_{s,k/n}^+$ is similar and uses 
the fact that $\lfloor a \rfloor = 0$ for $a \in [0,1)$.

\item First we show that the measure of the closure 
of the image of $\Upsilon_{s,\xi}^\pm$ is zero. As with 
$\Delta_s^\pm$, we consider the sizes of the gaps in 
this image. By the periodicity condition in part (ii), 
it is sufficient to consider the closure of the image 
of $\Upsilon_{s,\xi}$ over the interval $[0,1]$, which 
by part (iv) is contained in $[0,1]$. Note that a gap of 
size $(1 - s)s^{j-1}$ appears in the image due to 
the discontinuity of $\lceil x - \{j\xi\} \rceil$ 
in the $j$th term. Thus the complement of the closure 
of the image of $\Upsilon_{s,\xi}$ in $[0,1]$ has measure 
\[
\sum_{j=1}^\infty (1 - s)s^{j-1} = 1
\]
which implies that the closure of the image of 
$\Upsilon_{s,\xi}$ has measure $0$.

The proof that the closure of the image of 
$\Upsilon_{s,\xi}$ is a Cantor set follows the same 
reasoning as Theorem~\ref{T:delta}(xi).

\item This follows from the fact that each term defining 
$\Upsilon_{s,\upsilon}^\pm$ is continuous at $\xi$ 
as a function of $\upsilon$, together with uniform 
convergence of the series on compact intervals.

\item Let $x \in (0,1]$. By part (vi), 
\[
\Upsilon_{s,k/n}^-(x) 
= -\frac{1 - s}{s} + \frac{1 - s}{1 - s^n}
  \sum_{r=0}^{n-1} s^{r-1} 
  \left\lceil x - \left\{\frac{rk}{n}\right\}\right\rceil.
\]
As $s \to 1^-$, the first term approaches $0$ and the 
expression $(1 - s)/(1 - s^n)$ approaches $1/n$. Because 
$\gcd(k,n) = 1$, the set $\{rk\!\!\mod{n}\}_{r=0}^{n-1}$ is 
a permutation of $\{0,\dots,n-1\}$, and therefore 
\[
\lim_{s\to1^-}
\sum_{r=0}^{n-1} s^{r-1} 
\left\lceil x - \left\{\frac{rk}{n}\right\}\right\rceil 
= \sum_{r=0}^{n-1} \left\lceil x - \frac{r}{n} \right\rceil.
\]
Notice that every term of this final sum is equal to $0$ 
or $1$; the terms equaling $1$ are those for which 
$r/n < x$; there are $\lceil nx \rceil$ such terms. Thus 
$\sum_{r=0}^{n-1} \left\lceil x - \frac{r}{n} \right\rceil 
= \lceil nx \rceil$, which proves the statement for 
$x \in (0,1]$. 

The general result for $\Upsilon_{s,k/n}^-$ now follows from 
part (ii). 

The result for $\Upsilon_{s,k/n}^+$ follows immediately 
from the case $\Upsilon_{s,k/n}^-$ when $x \notin \frac1n\Z$, 
because for such $x$ we have the equality 
$\frac1n \lceil nx \rceil = \frac1n \lfloor nx+1 \rfloor$, 
and part (v) says that also $\Upsilon_{s,k/n}^+(x) 
= \Upsilon_{s,k/n}^-(x)$ for all $s$. When $x \in \frac1n\Z$, 
the result for $\Upsilon_{s,\xi}^+(x)$ follows from the right 
continuity of $\Upsilon_{s,\xi}^+$, as in part (iii).

\item The proof is similar to that of 
Theorem~\ref{T:delta}(xii). Note that, 
given any $x \in \R$, 
$\frac1n \lceil nx \rceil \to x$ and 
$\frac1n \lfloor nx+1 \rfloor \to x$
as $n \to \infty$. Thus, given any $\varepsilon > 0$, 
we can find $N$ large enough that $\Upsilon_{s,k/n}^+(x)$ 
is within $\varepsilon/2$ of $x$ whenever $n \ge N$ and 
$s$ is sufficiently close to $1$. By part (viii) we can 
ensure that $\Upsilon_{s,k/n}^+(x)$ is within 
$\varepsilon/2$ of $\Upsilon_{s,\xi}^+(x)$ whenever 
$k/n$ is sufficiently close to $\xi$, which in particular 
implies that $n$ must be large.
\qedhere
\end{enumerate}
\end{proof}

Parts (vi), (vii), and (viii) of Theorem~\ref{T:upsilon} 
are illustrated by Figure~\ref{F:cantor}. If $\xi\in\Q$, 
then the image of $\Upsilon_{s,\xi}^+$ is finite. If 
$\xi\notin\Q$, then the closure of the image of 
$\Upsilon_{s,\xi}^+$ is a Cantor set, and the image of 
$\Upsilon_{s,\upsilon}^+$ converges (in the Hausdorff 
metric) to this Cantor set as $\upsilon \to \xi$.

\begin{figure}[h]
\includegraphics[scale=.5]{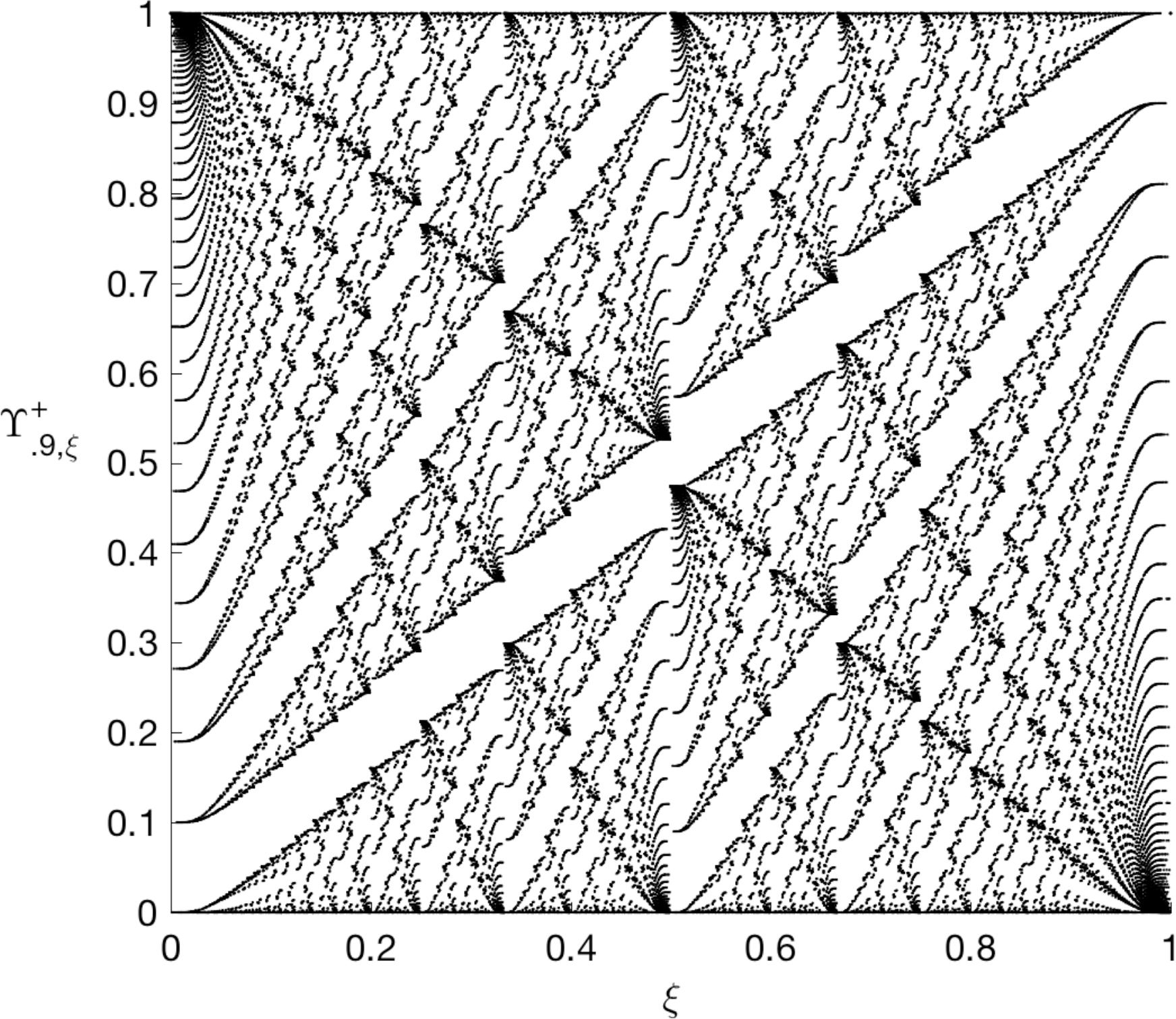}
\caption{The image of $\Upsilon_{s,\xi}^+$ for $0<\xi<1$.}
\label{F:cantor}
\end{figure}

Having seen that the Cantor sets appearing in 
Theorems \ref{T:delta}(xi) and \ref{T:upsilon}(vii) have 
measure $0$, we turn to their Hausdorff dimension.
Let $C_s$ be the closure of the image of $\Delta_s^\pm$, 
and let $C_{s,\xi}$ be the closure of the image of 
$\Upsilon_{s,\xi}^\pm$

\begin{thm}\label{T:hdim}
For all $s \in (0,1)$, $\xi > 0$, the Hausdorff 
dimension of either $C_s$ or $C_{s,\xi}$ is $0$.
\end{thm}

\begin{remark}
The Hausdorff dimension of $C_s$ is also shown to be $0$ in 
\cite[Theorem 4]{LN18}. The result for $C_{s,\xi}$ may be 
deduced from \cite[Theorem 7.3]{JansonOberg}.
\end{remark}

To prove Theorem~\ref{T:hdim}, we will use the notion of 
gap sums, as defined in \cite{BN96}. Here it will be helpful 
that we know precisely the sizes of the jumps that occur 
at discontinuities of the functions $\Delta_s^\pm$ and 
$\Upsilon_{s,\xi}^\pm$. 

Suppose $K \subset \R$ is compact, infinite, and has 
measure zero, and let $I_0 \subset \R$ be the convex 
hull of $K$ (that is, $I_0$ is the smallest closed 
interval in $\R$ that contains $K$). Then the complement 
of $K$ in $I_0$ is a collection of countably many open 
intervals $I_1, I_2, I_3, \dots$. Given $\delta > 0$, 
the \emph{degree $\delta$ gap sum} of $K$ is 
\[
S_\delta(K) = \sum_{j=1}^\infty |I_j|^\delta,
\]
where $|I_j|$ denotes the length of the $j$th interval 
in $I_0 \setminus K$. When $\delta = 1$, the gap sum 
equals the length of $I_0$, by our assumption that $K$ 
has measure zero. Hence the set of $\delta$ for which 
the series $S_\delta(K)$ converges is non-empty. 

In \cite{BT54} it is proved that the Hausdorff dimension 
of $K$ is at most the infimum of $\delta$ for which 
$S_\delta(K)$ converges. Therefore, to show that the 
Hausdorff dimension of $K$ is zero, it suffices to prove 
that $S_\delta(K)$ converges for all $\delta > 0$.

\begin{proof}[Proof of Theorem~\ref{T:hdim}]
Because $C_s$ is invariant under the translation 
$x \mapsto x + 1/s$, we will only consider the 
set $K_s = C_s \cap [0,1/s]$. We return to the 
calculation of the measure of the complement of $K_s$ 
in $[0,1/s]$ from the proof of Theorem~\ref{T:delta}(xi). 
There we saw that this complement contains $\varphi(n)$ 
intervals of length $s^{n-2}(1-s)^2/(1-s^n)$, so the 
degree $\delta$ gap sum for $K_s$ is 
\[
S_\delta(K_s) 
= \sum_{n=1}^\infty 
  \varphi(n) 
  \left(\frac{s^{n-2}(1 - s)^2}{1 - s^n}\right)^{\!\delta}
= \frac{(1 - s)^{2\delta}}{s^{2\delta}}
  \sum_{n=1}^\infty 
  \varphi(n) \left(\frac{s^n}{1 - s^n}\right)^{\!\delta}.
\]
The factor $(1 - s)^{2\delta}/s^{2\delta}$ does not 
affect the convergence of the series. Because $s^n \le s$ 
and $\varphi(n) \le n$ for all $n \ge 1$, we have 
\[
\varphi(n) \left(\frac{s^n}{1 - s^n}\right)^{\!\delta} 
\le \varphi(n) \frac{s^{n\delta}}{(1 - s)^\delta}
\le \frac{n (s^\delta)^n}{(1 - s)^\delta}.
\]
Again, the factor $1/(1 - s)^\delta$ does not affect 
convergence, so we just observe that $s^\delta < 1$, 
which implies that $\sum_{n=1}^\infty n(s^\delta)^n$ 
converges to $s^\delta/(1 - s^\delta)^2$, by 
Equation~\eqref{Eq:koebe}. Therefore, by the comparison 
test, $S_\delta(K_s)$ converges for all $\delta > 0$.

Because $C_{s,\xi}$ is invariant under the translation 
$x \mapsto x + 1$, we will only consider the set 
$K_{s,\xi} = C_{s,\xi} \cap [0,1]$. If $\xi\in\Q$, 
then $K_{s,\xi}$ is finite, and so its Hausdorff dimension 
is zero. If $\xi\notin\Q$, then the complement of 
$K_{s,\xi}$ in $[0,1]$ has one interval of length 
$(1-s)s^{j-1}$ for all $j \ge 1$, so the degree 
$\delta$ gap sum for $K_{s,\xi}$ is 
\[
S_\delta(K_{s,\xi}) 
= \sum_{j=1}^\infty (1 - s)^\delta s^{(j-1)\delta}.
\]
This is a geometric series with ratio $s^{\delta} < 1$, 
so it converges for all $\delta > 0$.
\end{proof}

The following lemma provides a crucial connection between 
the two kinds of functions we have defined in this section 
(cf.\ \cite[\S{II.2}]{Coutinho}).

\begin{lemma}\label{L:DU} 
For all $\xi > 0$ and for all $x \in [0,1)$
\[
\Upsilon_{s,\xi}^+(x+\xi) 
= s\big(\Upsilon_{s,\xi}^+(x) + \Delta_s^+(\xi)\big).
\]
\end{lemma}

\begin{proof}
Using the fact that $a = \lfloor a \rfloor + \{a\}$ 
for $a \in \R$, we have for any $x$ and for all $j \ge 1$
\begin{align*}
\lfloor x + \xi - \{j\xi\} \rfloor
&= \big\lfloor 
   x + \xi + \lfloor j\xi \rfloor - j\xi 
   \big\rfloor \\
&= \big\lfloor 
   x + \lfloor j\xi \rfloor - (j - 1)\xi 
   \big\rfloor \\
&= \big\lfloor 
   x + \lfloor j\xi \rfloor - \lfloor(j - 1)\xi\rfloor 
     - \{(j-1)\xi\}
   \big\rfloor \\
&= \lfloor x - \{(j-1)\xi\} \rfloor 
   + \lfloor j\xi \rfloor - \lfloor(j - 1)\xi\rfloor. 
\end{align*}
When $x \in [0,1)$ we also have 
$\big\lfloor x + \lfloor\xi\rfloor \big\rfloor 
= \lfloor x \rfloor + \lfloor\xi\rfloor 
= \lfloor \xi \rfloor$, and therefore in this case 
\begin{align*}
\sum_{j=1}^\infty s^j \lfloor x + \xi - \{j\xi\} \rfloor 
&= s \big\lfloor x + \lfloor\xi\rfloor \big\rfloor 
   + \sum_{j=2}^\infty s^j 
     \big( 
     \lfloor x - \{(j - 1)\xi\} \rfloor 
     + \lfloor j\xi \rfloor - \lfloor(j - 1)\xi\rfloor 
     \big) \\
&= s \lfloor \xi \rfloor 
   + \sum_{j=1}^\infty s^{j+1} \lfloor x - \{j\xi\} \rfloor
   + \sum_{j=2}^\infty s^j \lfloor j\xi \rfloor 
   - \sum_{j=1}^\infty s^{j+1} \lfloor j\xi\rfloor \\
&= s \sum_{j=1}^\infty s^j \lfloor x - \{j\xi\} \rfloor
   + (1 - s) \sum_{j=1}^\infty s^j \lfloor j\xi \rfloor. 
\end{align*}
Multiplying the first and last expressions by 
$(1 - s)/s$ and adding $1$ to both sides produces 
the equality 
\begin{align*}
&  1 + \frac{1 - s}{s} 
   \sum_{j=1}^\infty s^j \lfloor x + \xi - \{j\xi\} \rfloor \\
&= s + 
   s \left(\frac{1 - s}{s}\right) 
   \sum_{j=1}^\infty s^j \lfloor x - \{j\xi\} \rfloor 
   + (1 - s)
   + s \left(\frac{1 - s}{s}\right)^{\!2}
     \sum_{j=1}^\infty s^j \lfloor j\xi \rfloor \\
&= s
   \left( 
   1 + 
   \left(\frac{1 - s}{s}\right) 
   \sum_{j=1}^\infty s^j \lfloor x - \{j\xi\} \rfloor
   + \frac{1 - s}{s}
   + \left(\frac{1 - s}{s}\right)^{\!2}
     \sum_{j=1}^\infty s^j \lfloor j\xi \rfloor
   \right),
\end{align*}
or $\Upsilon_{s,\xi}^+(x + \xi) = 
s \big(\Upsilon_{s,\xi}^+(x) + \Delta_s^+(\xi)\big)$, 
as desired.
\end{proof}

The analogue of Lemma~\ref{L:DU} for 
$\Upsilon_{s,\xi}^-$ and $\Delta_s^-(\xi)$ 
follows in the case that $\xi \notin \Q$ 
from the facts that, in this case, 
$\Delta_s^-(\xi) = \Delta_s^+(\xi)$ and 
$\Upsilon_{s,\xi}^-(x) = \Upsilon_{s,\xi}^+(x)$ 
for a dense set of $x$ in $[0,1)$. In order to get 
an analogous result when $\xi$ is rational, however, 
we need to introduce a slight variant of 
$\Upsilon_{s,k/n}^-$: given positive integers $k$ and $n$ 
such that $\gcd(k,n) = 1$, define for $x \in (0,1]$ 
\[
\underline{\Upsilon}_{s,k/n}^-(x) 
= \frac{1 - s}{s(1 - s^n)}
  \sum_{r=1}^{n-1} s^r 
  \left\lceil x - \left\{\frac{rk}{n}\right\}\right\rceil, 
\]
The reader will observe that this function is closely 
analogous to the formula for $\Upsilon_{s,k/n}^+$ given 
in Theorem~\ref{T:upsilon}(vi), and that 
$\Upsilon_{s,k/n}^-(x) - \underline{\Upsilon}_{s,k/n}^-(x) 
= (s^{n-1} - s^n)/(1 - s^n)$. We leave it to the reader 
to prove that 
$\underline{\Upsilon}_{s,k/n}^-(x+k/n) = 
s\big(\underline{\Upsilon}_{s,k/n}^-(x) 
+ \Delta_s^-(k/n)\big)$ for all $x \in (0,1]$.

Lemma~\ref{L:DU} has the following dynamical consequence.

\begin{thm}\label{T:cantor}
Given $s \in (0,1)$ and $\xi > 0$, set $m = \Delta_s^+(\xi)$. 
As before, let $C_{s,\xi}$ be the closure of the image of 
$\Upsilon_{s,\xi}^+$, and set $K_{s,\xi} = 
C_{s,\xi} \cap [0,1]$. Define $f_{s,\xi} : [0,1) \to [0,1)$ 
by $f_{s,\xi}(y) = \{s(y+m)\}$. Then 
\[
K_{s,\xi} 
= \bigcap_{N=0}^\infty \overline{f_{s,\xi}^N\big([0,1)\big)}.
\]
Moreover, $f_{s,\xi} : K_{s,\xi} \to K_{s,\xi}$ is minimal 
if $\xi \notin \Q$.
\end{thm}

\begin{proof}
By Theorem~\ref{T:upsilon}(ii), 
$\big\{\Upsilon_{s,\xi}^+(x)\big\} 
= \Upsilon_{s,\xi}^+(\{x\})$ 
for all $x$. 
Lemma~\ref{L:DU} then implies that 
$f_{s,\xi}\big(\Upsilon_{s,\xi}^+(x)\big)
= \Upsilon_{s,\xi}^+\big(\{x+\xi\}\big)$
for all $x \in [0,1)$. By induction, 
\begin{equation}\label{Eq:fN}
f_{s,\xi}^N\big(\Upsilon_{s,\xi}^+(x)\big)
= \Upsilon_{s,\xi}^+\big(\{x+N\xi\}\big)
\end{equation}
for all $x \in [0,1)$ and $N \in \Z^+$.

Assume $y_0 \in K_{s,\xi}$. Then there exists a sequence 
$x_1,x_2,\dots$ of points in $[0,1)$ such that 
$\Upsilon_{s,\xi}^+(x_j)$ converges to $y_0$. If we let 
$x^*_N = \{x_N - N\xi\}$, then 
$f_{s,\xi}^N(\Upsilon_{s,\xi}^+(x_N^*)) = 
\Upsilon_{s,\xi}^+(x_N)$ by \eqref{Eq:fN}. Because 
the sets $f_{s,\xi}^N([0,1))$ are nested, we know that 
$y_0 \in \overline{f_{s,\xi}^N\big([0,1)\big)}$ for all $N$.

For the other inclusion, assume $y_0 \in 
\bigcap_{N=1}^\infty \overline{f_{s,\xi}^N\big([0,1)\big)}$. 
We observe that $f_{s,\xi}^N(0)$ is the upper endpoint of 
an open interval of size $(1-s)s^{N-1}$ in the complement of 
the image of $f_{s,\xi}^N$, and by assumption $y_0$ is not 
contained in any such open interval. Moreover, these intervals 
are all disjoint, and their total measure is $1$, so we can 
write 
\begin{align*}
y_0 
&= \sum_{f^N\!(0)\le y_0} (1 - s) s^{N-1} \\
&= \sum_{N=1}^\infty 
   (1 - s) s^{N-1} 
   \big( 1 + \lfloor y_0 - f^N(0) \rfloor \big) \\
&= \sum_{N=1}^\infty (1 - s) s^{N-1} 
   + (1 - s) 
     \sum_{N=1}^\infty s^{N-1} \lfloor y_0 - f^N(0) \rfloor \\
&= 1 
   + \frac{1-s}{s} 
     \sum_{N=1}^\infty s^N \lfloor y_0 - f^N(0) \rfloor.
\end{align*}
Now we need to find a sequence $x_1,x_2,\dots$ of points 
in $[0,1)$ such that $\Upsilon_{s,\xi}^+(x_\ell)$ converges 
to this value. Equivalently, we need the points $x_\ell$ to 
satisfy the limit
\[
\lim_{\ell\to\infty}
\sum_{N=1}^\infty s^N \lfloor x_\ell - \{N\xi\} \rfloor 
= \sum_{N=1}^\infty s^N \lfloor y_0 - f^N(0) \rfloor. 
\]
Fortunately, this equation tells us how to find such a 
sequence. Because $s$ is fixed and the coefficient of 
each $s^N$ in both series is either $0$ or $-1$, it is 
enough to choose $x_\ell$ so that the first $\ell$ 
coefficients of the series match. Given $\ell \ge 1$, 
let $x_\ell$ be the largest value of $\{N\xi\}$ such 
that  $0 \le N \le \ell$ and $f^N(0) \le y_0$. Then the 
two series above match in their first $\ell$ terms. 
Consequently, $\Upsilon_{s,\xi}^+(x_\ell)$ converges to 
$y_0$ as $\ell\to\infty$.

When $\xi \notin \Q$, the orbit of any point $x \in [0,1)$ 
is dense under the circle rotation $x \mapsto \{x+\xi\}$. 
Because $K_{s,\xi}$ is a perfect set, and the image of $[0,1)$ 
by $\Upsilon_{s,\xi}^+$ is dense in $K_{s,\xi}$, so is the 
image of the orbit of $x$. Thus the orbit of 
$\Upsilon_{s,\xi}^+(x)$ under $f_{s,\xi}$ is dense in 
$K_{s,\xi}$, which means that the restriction of 
$f_{s,\xi}$ to $K_{s,\xi}$ is minimal.
\end{proof}

When $\xi\notin\Q$, Theorem~\ref{T:cantor} can be interpreted 
in terms of AIETs as follows. Because $\Upsilon_{s,\xi}^+$ is 
strictly increasing, it can be inverted, and its inverse can 
be extended to a unique continuous non-decreasing function 
$g_{s,\xi}$ defined on all of $\R$. 
% It also satisfies the functional equation 
% $g_{s,\xi}(y+1) = g_{s,\xi}(y) + 1$. 
The following diagram commutes:
\[
\begin{CD}
[0,1)    @>{f_{s,\xi}}>> [0,1) \\
@V{g_{s,\xi}}VV @VV{g_{s,\xi}}V \\
[0,1)    @>>{x \mapsto \{x + \xi\}}> [0,1)
\end{CD}
\]
That is, $g_{s,\xi}$ is a semiconjugacy from the AIET 
$y \mapsto \{s(y + \Delta_s^+(\xi))\}$ to the circle 
rotation with parameter $\xi$. This function $g_{s,\xi}$ 
is a ``devil's staircase'' in the usual sense: it is a 
continuous, non-constant function that is almost everywhere 
locally constant. 

Similarly, $\Delta_s^-$ and $\Delta_s^+$ can be inverted, 
and their inverses can be extended to a continuous 
non-decreasing function which is also a devil's staircase.

% Theorems \ref{T:delta}, \ref{T:upsilon}, 
% \ref{T:hdim}, and \ref{T:cantor} 
The theorems of this section have dynamical implications for 
the linear trajectories on $X_s$, which we will explore after 
establishing a correspondence between such linear trajectories 
and those on the square torus $T^2 = \R^2/\Z^2$.

\section{Cutting sequences on $X_s^+$}\label{S:symbolic}

Let $s \in (0,1)$. In this section we focus on forward 
trajectories on $X_s^+$. Recall that ``forward'' means 
the trajectories locally move from left to right, and 
that $X_s^+$ is ``stable'' for such trajectories. In 
particular, a forward trajectory that starts on edge 
$A$ or $E$ (see Figure~\ref{F:Xs}) remains on $X_s^+$ 
for all (positive) time.

Our main goal for this section is to prove the following: 

\begin{thm}\label{T:cutting_sequences}
Let $\xi \ge 0$. 
\begin{enumerate}
[label*={(\roman*)}]
\item If $\xi \notin \Q$ and  $m = \Delta_s^\pm(\xi)$, 
then a forward trajectory with slope $m$ starting 
anywhere on edge $E$ has the same cutting sequence as 
a trajectory with slope $\xi$ on $T^2$.
\item If $\xi \in \Q$ and $m = \Delta_s^+(\xi)$, then 
the trajectory with slope $m$ starting at the top of edge 
$A$ is a saddle connection and has the same cutting sequence 
as a trajectory with slope $\xi$ on $T^2$.
\item If $\xi \in \Q$ and $m = \Delta_s^-(\xi)$, then 
the trajectory with slope $m$ starting at the bottom of edge 
$A$ is a saddle connection and has the same cutting sequence 
as a trajectory with slope $\xi$ on $T^2$.
\item If $\xi = k/n \in \Q$ with $\gcd(k,n)=1$, and 
$m \in (\Delta_s^-(\xi),\Delta_s^+(\xi))$, then a 
trajectory with slope $m$ starting at the point on edge 
$A$ at height 
\[
y_0 
= \frac{s^2}{1-s}m 
  - \frac{s^n}{1-s^n} 
    \sum_{\ell=1}^{k} 
    \left(\frac1s\right)^{\lfloor (\ell-1) n/k \rfloor}
\]
is closed and has the same cutting sequence as a 
trajectory with slope $\xi$ on $T^2$.
\end{enumerate}
\end{thm}

The restriction that $\xi \ge 0$ is not serious, because 
the affine automorphism $\psi_s$ of $X_s$ (from 
\S\ref{SS:mainexample}) transforms trajectories with 
negative slope into trajectories with positive slope and 
the same cutting sequences.
In addition, by applying a power of the Dehn twist $\phi_s$ 
(see again \S\ref{SS:mainexample}) that makes $0 \le m < 1/s$, 
we may assume that $0 \le \xi < 1$; this has the effect of 
allowing trajectories to cross the $A$ edge at most once in 
between crossings of the $B$ edge. 

To facilitate the proof of Theorem~\ref{T:cutting_sequences}, 
we introduce the notion of a stacking diagram, which is 
analogous to the unfolding diagram of a polygon reflected 
across its sides in sequence, as used in the study of 
polygonal billiards.

\subsection{Stacking diagrams}

Given a word $w = w_1 w_2 \dots$ in $A$s and $B$s, let 
$|w|$ denote its length (which may be finite or infinite). 
Also let $|w|_A$ be the number of times $A$ appears in $w$ 
and $|w|_B$ be the number of times $B$ appears in $w$, so 
that $|w| = |w|_A + |w|_B$. If $1 \le i \le |w|_B$, then 
let $\lambda_w(i)$ be the number of $A$s that appear before 
the $i$th $B$ in $w$. If $j \le |w|$, then set 
$w^j = w_1 w_2 \dots w_j$.

The \emph{stacking diagram} of $w$ is a union of $|w|+1$ 
rectangles (which we call ``boxes'') $R_j$ in $\R^2$, 
constructed in the following manner. 
\begin{itemize}
\item The $0$th box $R_0$ has vertices $(1-s,0)$, $(1,0)$, 
$(1,1)$, and $(1-s,1)$. 
\item If $w_j = A$, then the dimensions of $R_j$ are 
increased from those of $R_{j-1}$ by a factor of $1/s$, 
and $R_j$ is placed to the right of $R_{j-1}$, with their 
bottom edges aligned.
\item If $w_j = B$, then $R_j$ is placed on top of 
$R_{j-1}$, and the dimensions are unchanged.
\end{itemize}
Therefore, the horizontal sides of $R_j$ have length 
$s^{1-|w^j|_A}$, and the vertical sides of $R_j$ have 
length $s^{-|w^j|_A}$. The lower-left corner of $R_j$ 
is shifted from the lower-left corner of $R_{j-1}$ by
\begin{align*}
&\left(\frac{1}{s^{|w^j|_A-1}},\ 0\right)
&&\text{if $w_j = A$,} \\
&\left(0,\ \frac{1}{s^{|w^j|_A}}\right)
&&\text{if $w_j = B$.}
\end{align*}
Based on these data, we calculate that the upper right 
corner of $R_j$ has coordinates
\begin{equation}\label{Eq:Rjcoord}
(x_j,y_j) = 
\left(
  1 + \sum_{i=1}^{|w^j|_A} \frac{1}{s^{i-1}},\ 
  \frac{1}{s^{|w^j|_A}} 
    + \sum_{i=1}^{|w^j|_B} \frac{1}{s^{\lambda_w(i)}}
\right).
\end{equation}
Each box $R_j$ will be considered as a coordinate on 
$R_s^+$ (see Figure~\ref{F:Xs}).

\begin{figure}
\includegraphics{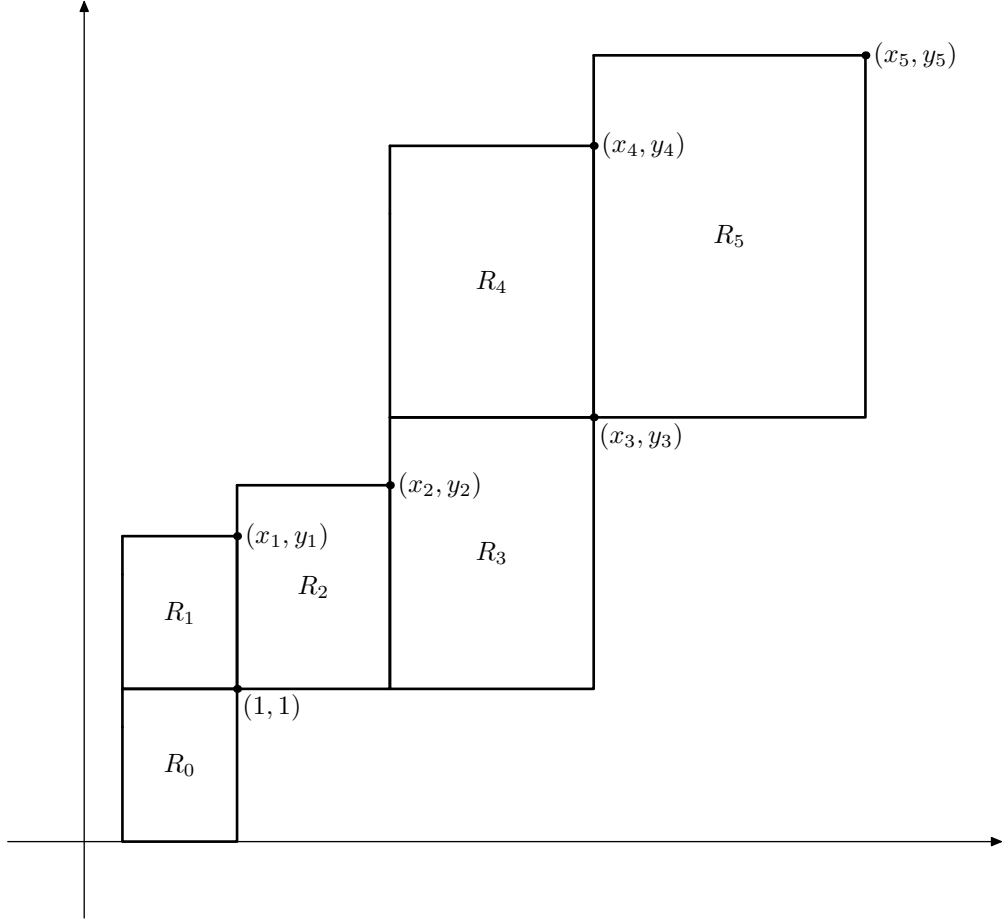}
\caption{Stacking diagram for the word $BAABA$.}
\label{F:stacking}
\end{figure}

\subsection{Canonical words}

Given a positive real number $\xi > 0$, we construct 
a canonical word $w = w(\xi)$, based on the cutting 
sequence of a trajectory with slope $\xi$ on the square 
torus $T^2 = X_1^+$. If $\xi\notin\Q$, then $w(\xi)$ is 
infinite, whereas if $\xi = k/n$ with $\gcd(k,n) = 1$, 
then $|w(\xi)| = k + n$. The case of $\xi \notin \Q$ 
produces what are known as \emph{Sturmian sequences}.

A trajectory that starts at $(0,0)$ with slope $\xi$ 
has equation $y = \xi x$. The $j$th square into which it 
enters (provided it has not passed through any additional 
corners besides $(0,0)$) is the square containing the 
intersection of the lines $y = \xi x$ and $x + y = j$. 
Because the $x$-coordinate of this intersection is 
$j/(\xi+1)$, after reaching the $j$th square the trajectory 
has crossed $\lfloor j/(\xi+1) \rfloor$ vertical edges. 
Therefore we set 
\begin{equation}\label{Eq:ABlength}
|w^j|_A = \left\lfloor \frac{j}{\xi + 1} \right\rfloor, 
\qquad
|w^j|_B = \left\lceil \frac{j\xi}{\xi + 1} \right\rceil,
\end{equation}
so that $|w^j|_A + |w^j|_B = j$. In particular, 
$|w^1|_A = \lfloor 1/(\xi + 1) \rfloor = 0$ and 
$|w^1|_B = \lceil \xi/(\xi + 1) \rceil = 1$, so $w_1 = B$. 
(This convention makes the tacit assumption, which will be 
useful later, that the trajectory ``entered'' the first 
square from the bottom---i.e., horizontal---edge.) When 
$\xi \notin \Q$, the sequences $|w^j|_A$ and $|w^j|_B$ are 
\emph{complementary Beatty sequences}.

To find the remaining letters of $w$, we check whether the 
number of $A$s or the number of $B$s increases when moving 
from $j-1$ to $j$: for $j > 1$,
\stepcounter{equation}
\begin{equation}\label{Eq:Acond}\tag{\arabic{equation}a}
w_j = 
\begin{cases}
A & \text{if}\quad 
    \lfloor j/(\xi+1) \rfloor
    - \lfloor (j-1)/(\xi+1) \rfloor 
    = 1, \\
B & \text{if}\quad 
    \lfloor j/(\xi+1) \rfloor
    - \lfloor (j-1)/(\xi+1) \rfloor 
    = 0.
\end{cases}
\end{equation}
Equivalently, 
\begin{equation}\label{Eq:Bcond}\tag{\arabic{equation}b}
w_j = 
\begin{cases}
A & \text{if}\quad 
    \lceil j\xi/(\xi+1) \rceil 
    - \lceil (j-1)\xi/(\xi+1) \rceil 
    = 0, \\
B & \text{if}\quad 
    \lceil j\xi/(\xi+1) \rceil 
    - \lceil (j-1)\xi/(\xi+1) \rceil 
    = 1.
\end{cases}
\end{equation}
\setcounter{equation}{13}
The number of $A$s that appear before the $i$th $B$ in 
$w(\xi)$ is given by the function 
\begin{equation}\label{Eq:lambdacount}
\lambda_w(i) 
= \left\lfloor \frac{i - 1}{\xi} \right\rfloor
\end{equation}
because this is the number of vertical edges crossed before 
the line $y = \xi x$ intersects $y = i - 1$. (We normalize 
so that $\lambda_w(1) = 0$ to match the convention that
$w_1 = B$.)

\subsection{Comparing coordinates: irrational case}
\label{irrational case}

Suppose $0 < \xi < 1$ and $\xi \notin \Q$. Set 
$m = \Delta_s^{\pm}(\xi)$. To prove 
Theorem~\ref{T:cutting_sequences}(i), we consider two 
particular trajectories: 
\begin{itemize}
\item $\tau_\xi^-$, which has slope $m$ and starts at the 
top of edge $E$;
\item $\tau_\xi^+$, which has slope $m$ and starts at the 
bottom of edge $E$.
\end{itemize}
We will develop both of these trajectories in the plane and 
show that the corresponding lines in $\R^2$ remain within 
the stacking diagram of $w = w(\xi)$. (See 
Figure~\ref{F:tauplusminus}. The reason for labeling the 
trajectories so that $\tau_\xi^-$ is ``above'' $\tau_\xi^+$ 
will become clear when we consider the rational case.) This 
will imply that a forward trajectory starting anywhere on 
$E$ with the same slope has the same cutting sequence $w$.

\begin{figure}
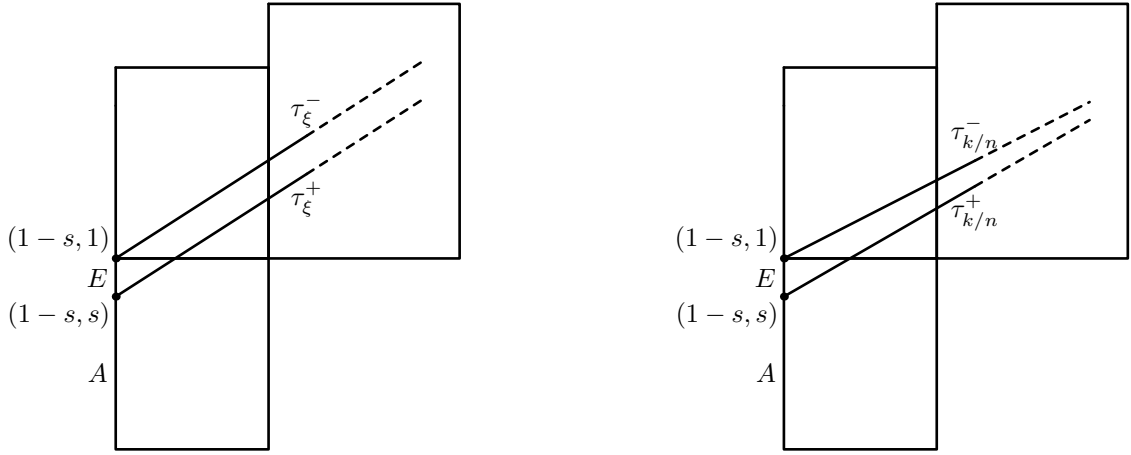

\includegraphics{surfacefig-4.mps}
\hspace{1in}
\includegraphics{surfacefig-5.mps}
\caption{The trajectories $\tau_\xi^-$ and $\tau_\xi^+$, 
drawn on a partial stacking diagram.
{\sc Left:} $\xi \notin \Q$, as in \S\ref{irrational case}. 
{\sc Right:} $\xi = k/n \in \Q$, as in \S\ref{rational case}}
\label{F:tauplusminus}
\end{figure}

Let $(x_j,y_j)$ be the coordinates of the upper right corner 
of the $j$th box in the stacking diagram of $w$. Combining 
Equations \eqref{Eq:Rjcoord}, \eqref{Eq:ABlength}, and 
\eqref{Eq:lambdacount}, we have 
\begin{align*}
x_j 
&= 1 + 
  \sum_{i=1}^{\lfloor j/(\xi+1) \rfloor} 
  \left(\frac{1}{s}\right)^{i-1}
 = 1 - s + \sum_{i=0}^{\lfloor j/(\xi+1) \rfloor} 
  \left(\frac{1}{s}\right)^{i-1}
 = 1 - s + s \sum_{i=0}^{\lfloor j/(\xi+1) \rfloor} 
  \left(\frac{1}{s}\right)^i \\
&= 1 - s 
  + s \cdot 
    \frac{(1/s)^{\lfloor j/(\xi+1) \rfloor + 1} - 1}{1/s - 1}
 = 1 - s 
  + \frac{s}{1 - s}
    \left(
    \left(\frac{1}{s}\right)^{\lfloor j/(\xi+1) \rfloor} 
    - s
    \right)
\end{align*}
and
\[
y_j 
= \left(\frac1s\right)^{\lfloor j/(\xi+1) \rfloor}
   + \sum_{i=1}^{\lceil j\xi/(\xi+1) \rceil}
     \left(\frac1s\right)^{\lfloor(i-1)/\xi\rfloor} 
= \left(\frac1s\right)^{\lfloor j/(\xi+1) \rfloor} 
   + \sum_{i=0}^{\lfloor j\xi/(\xi+1) \rfloor}
     \left(\frac1s\right)^{\lfloor i/\xi\rfloor}
\]
where we have used the fact that 
$\lceil j\xi/(\xi+1) \rceil - 1 
= \lfloor j\xi/(\xi+1) \rfloor$ 
because $\xi$ is irrational.

The development of $\tau_\xi^-$ into $\R^2$ with starting 
point $(1-s,1)$ is the line $y = m(x - 1 + s) + 1$. 
At $x_j$, the $y$-coordinate of this line (using the form 
of $\Delta_s^\pm(\xi)$ from Theorem~\ref{T:delta}(vi)) is 
\begin{align*}
y_j^-
&=\left(\frac{1-s}{s} 
  \sum_{\ell=0}^\infty s^{\lfloor \ell/\xi \rfloor}
  \right)\cdot
  \frac{s}{1 - s}
  \left(
  \left(\frac{1}{s}\right)^{\lfloor j/(\xi+1) \rfloor} 
   - s
  \right) + 1 \\
&= \left(
   \left(\frac{1}{s}\right)^{\lfloor j/(\xi+1) \rfloor} - s
   \right) \sum_{\ell=0}^\infty s^{\lfloor \ell/\xi \rfloor} 
   + 1
\end{align*}
Similarly, the development of $\tau_\xi^+$ into $\R^2$ 
starting at $(1-s,s)$ is the line $y = m(x - 1 + s) + s$. 
At $x_j$, the $y$-coordinate of this line is 
\begin{align*}
y_j^+ 
= y_j^- - 1 + s
= \left(
   \left(\frac{1}{s}\right)^{\lfloor j/(\xi+1) \rfloor} - s
   \right) \sum_{\ell=0}^\infty s^{\lfloor \ell/\xi \rfloor} 
   + s
\end{align*}

We wish to compare the value of $y_j$ with that of either 
$y_j^-$ or $y_j^+$, depending on whether $w_{j+1}$ is $A$ 
or $B$. To wit: 
\begin{itemize}
\item If $w_{j+1} = A$, then we should have $y_j^- < y_j$, 
meaning that the line $\tau_\xi^-$ crosses below the vertex 
$(x_j,y_j)$ in the stacking diagram for $w$. We prove this 
in \S\ref{Acrossing}.
\item if $w_{j+1} = B$, then we should have $y_j^+ > y_j$, 
meaning that the line $\tau_\xi^+$ crosses above the vertex 
$(x_j,y_j)$ in the stacking diagram for $w$. We prove this 
in \S\ref{Bcrossing}.
\end{itemize}
Let us illustrate with the cases $j = 0$ and $j = 1$. 
We already know that $w_1 = B$ and $(x_0,y_0) = (1,1)$; 
because $\xi > 0$ we have $m > (1-s)/s$, and so 
\[
y_0^+ = m(x_0 - 1 + s) + s = m(1 - 1 + s) + s > (1 - s) + s 
= y_0,
\]
as desired. At the next step, because we assume $\xi < 1$, 
we know that $1/2 < 1/(\xi+1) < 1$, and so $w_2 = A$ by 
Equation \eqref{Eq:Acond}. At the same time, 
$(x_1,y_1) = (1,2)$ and $m < 1/s$, so 
\[
y_1^- 
= m(x_1 - 1 + s) + 1 
= m(1 - 1 + s) + 1
< 1 + 1 = y_1;
\]
again the desired condition is met.

In \S\ref{Acrossing} and \S\ref{Bcrossing}, we will assume 
$j \ge 1$.

\subsubsection{Crossing an $A$ edge}\label{Acrossing}

Suppose $w_{j+1} = A$. Then the formulas in \eqref{Eq:Acond} 
and \eqref{Eq:Bcond}, along with the assumptions that 
$\xi \notin \Q$ and $j \ge 1$, respectively imply the 
following equalities:
\stepcounter{equation}
\begin{gather*}
\label{Eq:Aedgeconda}\tag{\arabic{equation}a}
\lfloor (j+1)/(\xi+1) \rfloor 
= \lfloor j/(\xi+1) \rfloor + 1, \\ 
\label{Eq:Aedgecondb}\tag{\arabic{equation}b}
\lfloor (j+1)\xi/(\xi+1) \rfloor 
= \lfloor j\xi/(\xi+1) \rfloor.
\end{gather*}
We will need both of these. 

We aim to show that $y_j > y_j^-$, 
or equivalently $y_j - y_j^- > 0$. Substituting the formulas 
that we found above for $y_j$ and $y_j^-$, we have 
\begin{align*}
y_j - y_j^- 
&= \left(\frac1s\right)^{\lfloor j/(\xi+1) \rfloor} 
   + \sum_{i=0}^{\lfloor j\xi/(\xi+1) \rfloor}
     \left(\frac1s\right)^{\lfloor i/\xi\rfloor}
   - \left(
   \left(\frac{1}{s}\right)^{\lfloor j/(\xi+1) \rfloor} - s
   \right) \sum_{\ell=0}^\infty s^{\lfloor \ell/\xi \rfloor} 
   - 1 \\
&= \left(\frac1s\right)^{\lfloor j/(\xi+1) \rfloor} 
   + \sum_{i=1}^{\lfloor j\xi/(\xi+1) \rfloor}
     \left(\frac1s\right)^{\lfloor i/\xi\rfloor}
   + \sum_{\ell=0}^\infty s^{\lfloor \ell/\xi \rfloor+1} 
   - \sum_{\ell=0}^\infty 
     s^{\lfloor\ell/\xi\rfloor - \lfloor j/(\xi+1) \rfloor} \\
&= \sum_{i=-\lfloor j\xi/(\xi+1) \rfloor}^{-1}
     s^{\lceil i/\xi\rceil}
   + \sum_{\ell=0}^\infty s^{\lfloor \ell/\xi \rfloor+1} 
   - \sum_{\ell=1}^\infty 
     s^{\lfloor\ell/\xi\rfloor - \lfloor j/(\xi+1) \rfloor} \\
&= \sum_{\ell=-\lfloor j\xi/(\xi+1) \rfloor}^\infty 
     s^{\lfloor \ell/\xi \rfloor+1} 
   - \sum_{\ell=1}^\infty 
     s^{\lfloor\ell/\xi\rfloor - \lfloor j/(\xi+1) \rfloor}, 
\end{align*}
where we have used the facts that $\lceil-a\rceil = 
-\lfloor a\rfloor$ for all $a \in \R$ and 
$\lceil i/\xi \rceil = \lfloor i/\xi \rfloor + 1$ for all 
nonzero integers $i$ since $\xi \notin \Q$.

Start with the inequality
\[
\frac{(j+1)\xi}{\xi+1} < 
\left\lfloor \frac{(j+1)\xi}{\xi+1} \right\rfloor + 1,
\]
which is true by the definition of the floor function. 
Using \eqref{Eq:Aedgecondb}, this is equivalent to 
\[
\frac{(j+1)\xi}{\xi+1} < 
\left\lfloor \frac{j\xi}{\xi+1} \right\rfloor + 1.
\]
Next we subtract both sides from $\ell$ and divide by 
$\xi$ to get 
\[
\frac{1}{\xi}
\left(
\ell - \left\lfloor \frac{j\xi}{\xi+1} \right\rfloor - 1
\right) < 
\frac{\ell}{\xi} - \frac{j+1}{\xi+1} 
\qquad\text{for all $\ell \ge 1$.}
\]
Applying the floor function to each side, we obtain 
\[
\left\lfloor \frac{1}{\xi}
   \left(\ell 
   - \left\lfloor \frac{j\xi}{\xi+1} \right\rfloor - 1\right)
   \right\rfloor \le 
\left\lfloor \frac{\ell}{\xi} - \frac{j+1}{\xi+1} \right\rfloor
\qquad\text{for all $\ell \ge 1$.}
\]
Now, if $a,b \in \R$, it is true that 
$\lfloor a - b \rfloor \le 
\lfloor a \rfloor - \lfloor b \rfloor$, and so 
\[
\left\lfloor \frac{1}{\xi}
   \left(\ell 
   - \left\lfloor \frac{j\xi}{\xi+1} \right\rfloor - 1\right)
   \right\rfloor \le 
\left\lfloor \frac{\ell}{\xi} \right\rfloor 
- \left\lfloor \frac{j+1}{\xi+1} \right\rfloor
\qquad\text{for all $\ell \ge 1$.}
\]
Moreover, the preceding equality is strict except possibly when 
$\{\ell/\xi\} = \{(j+1)/(\xi+1)\}$, which is true for at most 
one value of $\ell$. Using \eqref{Eq:Aedgeconda}, this last 
inequality is equivalent to
\[
\left\lfloor \frac{1}{\xi}
   \left(\ell 
   - \left\lfloor \frac{j\xi}{\xi+1} \right\rfloor - 1\right)
   \right\rfloor + 1 \le 
\left\lfloor \frac{\ell}{\xi} \right\rfloor 
- \left\lfloor \frac{j}{\xi+1} \right\rfloor
\qquad\text{for all $\ell \ge 1$.}
\]
Because $0 < s < 1$, this implies 
\[
s^{\left\lfloor 
   (\ell - \lfloor j\xi/(\xi+1) \rfloor - 1)/\xi 
   \right\rfloor + 1} \ge
s^{\lfloor \ell/\xi \rfloor - \lfloor j/(\xi+1) \rfloor} 
\qquad\text{for all $\ell \ge 1$.}
\]
Therefore 
\[
\sum_{\ell=-\lfloor j\xi/(\xi+1)\rfloor}^{\infty} 
  s^{\lfloor \ell/\xi \rfloor+1} > 
\sum_{\ell=1}^\infty 
  s^{\lfloor \ell/\xi \rfloor 
  - \lfloor j/(\xi+1) \rfloor}
\]
because the series on the left is term-by-term greater than 
or equal to the series on the right, with all terms except 
possibly one being strictly greater. This last inequality 
is equivalent to $y_j - y_j^- > 0$, and so the proof that 
$y_j > y_j^-$ when $w_{j+1} = A$ is complete.

% \subsubsection*{Extra stuff for now (from earlier draft)}

% Fortunately, the expressions for $y_j$ and $y_j'$ have 
% the same first term, so we just need to establish that 
% \[
% \sum_{i=0}^{\lfloor j\xi/(\xi+1) \rfloor}
%      \left(\frac1s\right)^{\lfloor i/\xi\rfloor} 
% - \left(
%     \left(\frac{1}{s}\right)^{\lfloor j/(\xi+1) \rfloor} 
%     - s
%     \right) \sum_{\ell=1}^\infty s^{\lfloor \ell/\xi \rfloor}
% \begin{cases}
% > 0 & \text{when crossing an $A$ edge} \\
% < 0 & \text{when crossing a $B$ edge}
% \end{cases}
% \]
% First note that 
% \begin{align*}
% \sum_{i=0}^{\lfloor j\xi/(\xi+1) \rfloor}
%      \left(\frac1s\right)^{\lfloor i/\xi\rfloor} 
% + s \sum_{\ell=1}^\infty s^{\lfloor \ell/\xi \rfloor}
% &= 1 
%    + \sum_{i=-\lfloor j\xi/(\xi+1)\rfloor}^{-1} 
%      s^{\lceil i/\xi \rceil}
%    + \sum_{\ell=1}^\infty s^{\lfloor \ell/\xi \rfloor + 1} \\
% &= 1 - s
%    + \sum_{i=-\lfloor j\xi/(\xi+1)\rfloor}^{-1} 
%      s^{\lfloor i/\xi \rfloor + 1}
%    + \sum_{\ell=0}^\infty s^{\lfloor \ell/\xi \rfloor + 1} \\
% &= 1 - s
%    + \sum_{\ell=-\lfloor j\xi/(\xi+1)\rfloor}^{\infty} 
%      s^{\lfloor \ell/\xi \rfloor + 1} \\
% \end{align*}
% so the desired condition becomes 
% \begin{equation}\label{Eq:compare}
% 1 - s
%    + \sum_{\ell=-\lfloor j\xi/(\xi+1)\rfloor}^{\infty} 
%      s^{\lfloor \ell/\xi \rfloor + 1}
%     - \sum_{\ell=1}^\infty 
%       s^{\lfloor \ell/\xi \rfloor - \lfloor j/(\xi+1) \rfloor}
% \begin{cases}
% > 0 & \text{when crossing an $A$ edge} \\
% < 0 & \text{when crossing a $B$ edge}
% \end{cases}
% \end{equation}

\subsubsection{Crossing a $B$ edge}\label{Bcrossing}

Suppose $w_{j+1} = B$. From formulas \eqref{Eq:Acond} 
and \eqref{Eq:Bcond} we obtain:
\stepcounter{equation}
\begin{gather*}
\label{Eq:Bedgeconda}\tag{\arabic{equation}a}
\lfloor (j+1)/(\xi+1) \rfloor 
= \lfloor j/(\xi+1) \rfloor, \\
\label{Eq:Bedgecondb}\tag{\arabic{equation}b}
\lfloor (j+1)\xi/(\xi+1) \rfloor 
= \lfloor j\xi/(\xi+1) \rfloor + 1.
\end{gather*}

We want to show that $y_j^+ > y_j$; because 
$y_j^+ = y_j^- - 1 + s$, this is equivalent to showing 
\[
\sum_{\ell=1}^\infty 
s^{\lfloor\ell/\xi\rfloor-\lfloor j/(\xi+1)\rfloor}
- \sum_{\ell=-\lfloor j\xi/(\xi+1)\rfloor}^\infty
  s^{\lfloor\ell/\xi\rfloor+1}
> 1 - s.
\]
Substitute using \eqref{Eq:Bedgeconda} into the exponent 
of each term in the first sum and  \eqref{Eq:Bedgecondb} 
into the bottom index of the second sum, then re-index so 
that the left side becomes 
\begin{align*}
& \sum_{\ell=1}^\infty 
   s^{\left\lfloor\frac\ell\xi\right\rfloor
   -\left\lfloor\frac{j+1}{\xi+1}\right\rfloor}
   - \sum_{\ell=1-\lfloor (j+1)\xi/(\xi+1)\rfloor}^\infty
   s^{\left\lfloor\frac\ell\xi\right\rfloor+1} \\
=\ & \sum_{\ell=1}^\infty 
   s^{\left\lfloor\frac\ell\xi\right\rfloor
   -\left\lfloor\frac{j+1}{\xi+1}\right\rfloor}
   - \sum_{\ell=1}^\infty
   s^{\left\lfloor\frac\ell\xi
   -\frac1\xi
    \left\lfloor\frac{(j+1)\xi}{\xi+1}\right\rfloor
    \right\rfloor+1} \\
=\ & \sum_{\ell=1}^\infty 
   \left(s^{\left\lfloor\frac\ell\xi\right\rfloor
   -\left\lfloor\frac{j+1}{\xi+1}\right\rfloor}
   - s^{\left\lfloor\frac\ell\xi
   -\frac1\xi
    \left\lfloor\frac{(j+1)\xi}{\xi+1}\right\rfloor
    \right\rfloor+1}
\right)
\end{align*}
To ease notation in the rest of this section, let 
$\sigma$ be this last sum, and set $j' = j+1$, so that 
\[
\sigma 
= \sum_{\ell=1}^\infty 
  \left(
  s^{\left\lfloor\frac\ell\xi\right\rfloor
  - \left\lfloor\frac{j'}{\xi+1}\right\rfloor}
  - s^{\left\lfloor\frac\ell\xi
    -\frac1\xi
    \left\lfloor\frac{j'\xi}{\xi+1}\right\rfloor
    \right\rfloor+1}
  \right).
\]
Each term in $\sigma$ is non-negative by the following 
lemma, where $a=\ell$, $b=\xi$, and $c=j'/(\xi+1)$.

\begin{lemma}\label{L:floorineq}
For all $a, b, c > 0$, 
\begin{equation}\label{Eq:floorineq}
\left\lfloor\frac{a}{b}\right\rfloor-\left\lfloor c\right\rfloor \le \left\lfloor\frac{a}{b}-\frac{\left\lfloor bc\right\rfloor}{b}\right\rfloor +1.
\end{equation}
A sufficient condition for \eqref{Eq:floorineq} to be strict is 
\begin{equation}\label{Eq:fracineq}
\left\{\frac{a}{b}\right\}+\frac{\left\{bc\right\}}{b} 
\ge \{c\},
\end{equation}
and whenever \eqref{Eq:fracineq} is satisfied with equality, 
we have
\begin{equation}\label{Eq:flooreq}
\left\lfloor\frac{a}{b}\right\rfloor-\left\lfloor c\right\rfloor
=\left\lfloor\frac{a}{b}-\frac{\left\lfloor bc\right\rfloor}{b}\right\rfloor.
\end{equation}
\end{lemma}
\begin{proof}
Start with 
\[\left\lfloor\frac{a}{b}\right\rfloor= \left\lfloor\frac{a}{b}-\frac{bc}{b}+c\right\rfloor \le \left\lfloor\frac{a}{b}-\frac{\left\lfloor bc\right\rfloor}{b}+c\right\rfloor \le \left\lfloor\frac{a}{b}-\frac{\left\lfloor bc\right\rfloor}{b}\right\rfloor+\left\lfloor c\right\rfloor+1.
\]
Now subtract $\lfloor c \rfloor$ from the first and last 
expressions to obtain \eqref{Eq:floorineq}.  

Next, notice that 
\begin{align*}
\left\lfloor\frac{a}{b}
- \frac{\left\lfloor bc\right\rfloor}{b}\right\rfloor 
&=\left\lfloor\frac{a}{b}-\frac{bc-\{bc\}}{b}\right\rfloor \\
&=\left\lfloor\left(\left\lfloor\frac{a}{b}\right\rfloor
+\left\{\frac{a}{b}\right\}\right)-(\lfloor c\rfloor
+\{ c\})+\frac{\{ bc\}}{b}\right\rfloor \\
&=\left\lfloor\frac{a}{b}\right\rfloor
-\lfloor c\rfloor+\left\lfloor\left\{\frac{a}{b}\right\}
-\{ c\}+\frac{\{ bc\}}{b}\right\rfloor 
\end{align*}
so that \eqref{Eq:floorineq} is strict if \eqref{Eq:fracineq} is satisfied, and whenever $\left\{\frac{a}{b}\right\}+\frac{\left\{bc\right\}}{b}-\left\{c\right\}=0$, equality \eqref{Eq:flooreq} also holds.
\end{proof}

Recall that we can assume $0 < \xi < 1$, thanks to the 
affine group of $X_s$, so let $\xi = [0;a_1,a_2,a_3,\dots]$ 
be the (infinite) continued fraction of $\xi$. As in the 
statement of Lemma~\ref{convergent_formulas}, let 
$P_i/Q_i$ be the convergents of $\xi$, $P'_i/Q'_i$ the 
convergents of $1/(\xi+1)$, and $P''_i/Q''_i$ the 
convergents of $\xi/(\xi+1)$.

Our goal now is to show that $\sigma > 1 - s$. 
Lemma~\ref{L:floorineq} suggests that we want to find 
values of $\ell$ such that 
\begin{equation}\label{remainders_ineq}
\bigg\{\frac\ell\xi\bigg\}
    +\frac1\xi\bigg\{\frac{j'\xi}{\xi+1}\bigg\}\ge \bigg\{\frac{j'}{\xi+1}\bigg\}
\end{equation}
so that we can consider only positive (i.e., nonzero) 
terms in the expression for $\sigma$.  We find the 
desired values of $\ell$ with the help of 
Lemma~\ref{convergent_formulas}. We consider two cases: 
\begin{itemize}
\item[(I)] $j'$ is the denominator of an intermediate fraction 
of $1/(\xi+1)$, such that $\{j'/(\xi+1)\}$ is a near approach 
to $1$ (see section \ref{Continued fractions}); or
\item[(II)] $j'$ is any other value.
\end{itemize}

In case (I), we suppose 
$j' = P_{2i-2} + Q_{2i-2} + \alpha (P_{2i-1} + Q_{2i-1})$ 
for some $0\le\alpha\le a_{2i}$. By 
Lemma~\ref{convergent_formulas}, this is the same as 
\[
j' = Q'_{2i-1} + \alpha Q'_{2i}
= Q''_{2i-2} + \alpha Q''_{2i-1}.
\]
Let $\ell=P_{2i-2}+\alpha P_{2i-1}$.  
The left side of \eqref{remainders_ineq} becomes 
\begin{align*}
\bigg\{\frac\ell\xi\bigg\}
    +\frac1\xi\bigg\{\frac{j'\xi}{\xi+1}\bigg\}
    &=\bigg\{\frac{P_{2i-2}+\alpha P_{2i-1}}{\xi}\bigg\}
+\frac{Q'_{2i-1}+\alpha Q'_{2i}}{\xi+1} -\frac1\xi\left\lfloor\frac{(Q''_{2i-2}+\alpha Q''_{2i-1})\xi}{\xi+1}\right\rfloor\\
&=\bigg\{\frac{P_{2i-2}+\alpha P_{2i-1}}{\xi}\bigg\}
+\frac{Q'_{2i-1}+\alpha Q'_{2i}}{\xi+1} -\frac{P''_{2i-2}+\alpha P''_{2i-1}}{\xi}\\
&=\frac{Q'_{2i-1}+\alpha Q'_{2i}}{\xi+1} -\left\lfloor\frac{P_{2i-2}+\alpha P_{2i-1}}{\xi}\right\rfloor  
\end{align*}
by relation \eqref{int_frac_floor} and the fact that $P''_i=P_i$.  Since 
$\frac{P_{2i-2}+\alpha P_{2i-1}}{Q_{2i-2}+\alpha Q_{2i-1}}
<\xi$, we have 
$\frac{P_{2i-2}+\alpha P_{2i-1}}{\xi} 
< Q_{2i-2}+\alpha Q_{2i-1}$ and 
$\left\lfloor
 \frac{P_{2i-2}+\alpha P_{2i-1}}{\xi}
 \right\rfloor = Q_{2i-2}+\alpha Q_{2i-1}-1$. 
But $Q_{2i-2}+\alpha Q_{2i-1}-1 = 
P'_{2i-1}+\alpha P'_{2i}-1 = 
\left\lfloor
\frac{Q'_{2i-1}+\alpha Q'_{2i}}{\xi+1}
\right\rfloor$ by relation \eqref{int_frac_floor}.  
We then see that 
\[
\frac{Q'_{2i-1}+\alpha Q'_{2i}}{\xi+1} -\left\lfloor\frac{P_{2i-2}+\alpha P_{2i-1}}{\xi}\right\rfloor=\frac{Q'_{2i-1}+\alpha Q'_{2i}}{\xi+1} -\left\lfloor\frac{Q'_{2i-1}+\alpha Q'_{2i}}{\xi+1}\right\rfloor=\bigg\{\frac{Q'_{2i-1}+\alpha Q'_{2i}}{\xi+1}\bigg\},
\]
so \eqref{remainders_ineq} is satisfied with equality 
if $\ell=P_{2i-2}+\alpha P_{2i-1}$ and 
$j' = P_{2i-2}+Q_{2i-2}+\alpha(P_{2i-1}+Q_{2i-1})$. 
By Lemma \ref{L:floorineq}, because \eqref{remainders_ineq} 
is satisfied with equality, we have 
\[
\left\lfloor\frac\ell\xi
   -\frac1\xi
    \left\lfloor\frac{j'\xi}{\xi+1}\right\rfloor
    \right\rfloor+1=\left\lfloor\frac\ell\xi\right\rfloor
   -\left\lfloor\frac{j'}{\xi+1}\right\rfloor+1.
\]
Since $\xi\notin\mathbb{Q}$, $\big\{\frac{\ell}{\xi}\big\}$ 
can be made as close to $1$ as we like so long as $\ell$ 
can be chosen sufficiently large.  That is, for some 
$\ell \gg P_{2i-2}+\alpha P_{2i-1}$, we have 
$\big\{\frac{\ell}{\xi}\big\} > 
\big\{\frac{P_{2i-2}+\alpha P_{2i-1}}{\xi}\big\}$.  
In such a case, inequality \eqref{remainders_ineq} is 
strict. 

With these findings, we see for 
$j' = P_{2i-2}+Q_{2i-2}+\alpha(P_{2i-1}+Q_{2i-1})$ that
\begin{align*}
% &\sum_{\ell=1}^\infty 
% \left(s^{\left\lfloor\frac\ell\xi\right\rfloor
%    -\left\lfloor\frac{j+1}{\xi+1}\right\rfloor}
% - s^{\left\lfloor\frac\ell\xi
%    -\frac1\xi
%     \left\lfloor\frac{(j+1)\xi}{\xi+1}\right\rfloor
%     \right\rfloor+1}
% \right)\\
\sigma &\ge \sum_{\ell=P_{2i-2}+\alpha P_{2i-1}}^\infty 
\left(s^{\left\lfloor\frac\ell\xi\right\rfloor
   -\left\lfloor\frac{j'}{\xi+1}\right\rfloor}
- s^{\left\lfloor\frac\ell\xi
   -\frac1\xi
    \left\lfloor\frac{j'\xi}{\xi+1}\right\rfloor
    \right\rfloor+1}
\right)\\
&= (1-s)
   s^{
      \left\lfloor
      \frac{P_{2i-2}+\alpha P_{2i-1}}{\xi}
      \right\rfloor
      - \left\lfloor
        \frac{Q'_{2i-1}+\alpha Q'_{2i}}{\xi+1}
        \right\rfloor}
   + \sum_{\ell=P_{2i-2}+\alpha P_{2i-1}+1}^\infty 
   \left(s^{\left\lfloor\frac\ell\xi\right\rfloor
   -\left\lfloor\frac{j'}{\xi+1}\right\rfloor}
- s^{\left\lfloor\frac\ell\xi
   -\frac1\xi
    \left\lfloor\frac{j'\xi}{\xi+1}\right\rfloor
    \right\rfloor+1}
\right)\\
&= (1-s)
   + \sum_{\ell=P_{2i-2}+\alpha P_{2i-1}+1}^\infty 
     \left(
     s^{\left\lfloor\frac\ell\xi\right\rfloor
        - \left\lfloor\frac{j'}{\xi+1}\right\rfloor}
     - s^{\left\lfloor\frac\ell\xi
          -\frac1\xi
          \left\lfloor\frac{j'\xi}{\xi+1}\right\rfloor
          \right\rfloor+1}
   \right)
\end{align*}
since 
$\left\lfloor
 \frac{P_{2i-2}+\alpha P_{2i-1}}{\xi}
 \right\rfloor
 = \left\lfloor
   \frac{Q'_{2i-1}+\alpha Q'_{2i}}{\xi+1}
   \right\rfloor$, 
as shown above. We also saw that for some 
$\ell \gg P_{2i-2}+\alpha P_{2i-1}$, 
inequality \eqref{remainders_ineq} is strict, which 
implies that the final summation term above is positive.  
We now know that $\sigma > 1-s$ when 
$j'=P_{2i-2}+Q_{2i-2}+\alpha(P_{2i-1}+Q_{2i-1})$.  

As we recalled above, 
$\big\{\frac{P_{2i-2}+Q_{2i-2}
 + \alpha(P_{2i-1}+Q_{2i-1})}{\xi+1}\big\}
 = \big\{\frac{Q'_{2i-1}+\alpha Q'_{2i}}{\xi+1}\big\}$ 
describes every near approach to $1$. This implies that 
for all $j' \neq Q'_{2i-1}+\alpha Q'_{2i}$, either 
$\big\{\frac{j'}{\xi+1}\big\}
< \big\{\frac{1}{\xi+1}\big\}$ or 
$\big\{\frac{Q'_{2i-1}+(\alpha-1)Q'_{2i}}{\xi+1}\big\}
< \big\{\frac{j'}{\xi+1}\big\}
< \big\{\frac{Q'_{2i-1}+\alpha Q'_{2i}}{\xi+1}\big\}$ 
for some $i\ge 1$ and $1\le\alpha\le a_{2i}$.  With the 
help of Lemma \ref{(j+1)BtwnIntFracs}, we can now show 
that $\sigma > 1-s$ also in case (II).
\begin{lemma}\label{(j+1)BtwnIntFracs}
If $\big\{\frac{j'}{\xi+1}\big\}<\big\{\frac{1}{\xi+1}\big\}$ 
or $\big\{\frac{Q'_{2i-1}+(\alpha-1)Q'_{2i}}{\xi+1}\big\}
< \big\{\frac{j'}{\xi+1}\big\}
< \big\{\frac{Q'_{2i-1}+\alpha Q'_{2i}}{\xi+1}\big\}$ 
for some \(i\ge 1\) and \(1\le\alpha\le a_{2i}\), then 
\[
\bigg\{\frac\ell\xi\bigg\}+\frac1\xi\bigg\{\frac{j'\xi}{\xi+1}\bigg\}>\bigg\{\frac{j'}{\xi+1}\bigg\}
\]
and
\[
\left\lfloor
\frac\ell\xi\right\rfloor - \left\lfloor\frac{j'}{\xi+1}
\right\rfloor<0
\]
for \(\ell=P_{2i-2}+\alpha P_{2i-1}\).
\end{lemma}
\begin{proof}
If 
$\big\{\frac{Q'_{2i-1}+(\alpha-1)Q'_{2i}}{\xi+1}\big\}
<\big\{\frac{j'}{\xi+1}\big\}
<\big\{\frac{Q'_{2i-1}+\alpha Q'_{2i}}{\xi+1}\big\}$, 
then
\begin{align*}
\bigg\{\frac\ell\xi\bigg\}
+\frac1\xi\bigg\{\frac{j'\xi}{\xi+1}\bigg\}
&=\bigg\{\frac{P_{2i-2}+\alpha P_{2i-1}}{\xi}\bigg\}
  +\frac1\xi\left(1-\bigg\{\frac{j'}{\xi+1}\bigg\}\right)\\
&>\bigg\{\frac{P_{2i-2}+\alpha P_{2i-1}}{\xi}\bigg\}
  +\frac1\xi
   \left(1-\bigg\{\frac{Q'_{2i-1}+\alpha Q'_{2i}}{\xi+1}\bigg\}\right)\\
&=\bigg\{\frac{P_{2i-2}+\alpha P_{2i-1}}{\xi}\bigg\}
 +\frac1\xi
  \bigg\{\frac{(Q'_{2i-1}+\alpha Q'_{2i})\xi}{\xi+1}\bigg\}\\
&=\bigg\{\frac{P_{2i-2}+\alpha P_{2i-1}}{\xi}\bigg\}
 +\frac1\xi
  \bigg\{\frac{(Q''_{2i-2}+\alpha Q''_{2i-1})\xi}{\xi+1}\bigg\}\\
&=\bigg\{\frac{Q'_{2i-1}+\alpha Q'_{2i}}{\xi+1}\bigg\}\\
&>\bigg\{\frac{j'}{\xi+1}\bigg\}.
\end{align*}
If $\big\{\frac{j'}{\xi+1}\big\}<\big\{\frac{1}{\xi+1}\big\}$, 
then also 
$\big\{\frac{j'}{\xi+1}\big\}<\big\{\frac{Q'_{2i-1}
 +\alpha Q'_{2i}}{\xi+1}\big\}$ 
and the proof is the same as above.  

Next, notice that 
$\big\{\frac{Q'_{2i-1}+(\alpha-1)Q'_{2i}}{\xi+1}\big\}$ and 
$\big\{\frac{Q'_{2i-1}+\alpha Q'_{2i}}{\xi+1}\big\}$ 
are successive near approaches to 1, so for 
$\big\{\frac{Q'_{2i-1}+(\alpha-1)Q'_{2i}}{\xi+1}\big\}
 <\big\{\frac{j'}{\xi+1}\big\}
 <\big\{\frac{Q'_{2i-1}+\alpha Q'_{2i}}{\xi+1}\big\}$, 
it must be true that $j' > Q'_{2i-1}+\alpha Q'_{2i}$ and 
$\left\lfloor\frac{j'}{\xi+1}\right\rfloor
 >\left\lfloor\frac{Q'_{2i-1}+\alpha Q'_{2i}}{\xi+1}\right\rfloor$.  Then 
\[
\left\lfloor\frac\ell\xi\right\rfloor
-\left\lfloor\frac{j'}{\xi+1}\right\rfloor
< \left\lfloor\frac{P_{2i-2}+\alpha P_{2i-1}}{\xi}\right\rfloor
  -\left\lfloor\frac{Q'_{2i-1}+\alpha Q'_{2i}}{\xi+1}\right\rfloor
=0.
\]
Similarly, if 
$\big\{\frac{j'}{\xi+1}\big\}<\big\{\frac{1}{\xi+1}\big\}$, 
then also $\left\lfloor\frac{j'}{\xi+1}\right\rfloor
> \left\lfloor\frac{1}{\xi+1}\right\rfloor$; the rest of the proof is identical.
\end{proof}

By Lemma \ref{(j+1)BtwnIntFracs}, if 
$j'\neq P_{2i-2}+Q_{2i-2}+\alpha(P_{2i-1}+Q_{2i-1})$, then 
\[
\sigma 
\ge (1-s)
    s^{\left\lfloor
       \frac{P_{2i-2}+\alpha P_{2i-1}}{\xi}
       \right\rfloor
       -\left\lfloor\frac{j'}{\xi+1}\right\rfloor}
> 1-s.
\]
We have now shown in both case (I) and case (II)---that is, 
for any $j' \ge 1$---that 
\[
\sum_{\ell=1}^\infty 
\left(s^{\left\lfloor\frac\ell\xi\right\rfloor
   -\left\lfloor\frac{j+1}{\xi+1}\right\rfloor}
- s^{\left\lfloor\frac\ell\xi
   -\frac1\xi
    \left\lfloor\frac{(j+1)\xi}{\xi+1}\right\rfloor
    \right\rfloor+1}
\right)>1-s,
\]
as desired. This completes the proof that 
$y_j < y_j^+$ when $w_{j+1} = B$.

\subsection{Comparing coordinates: rational case}
\label{rational case}

Now, suppose $0\le\xi<1$ with $\xi=k/n\in\mathbb{Q}$ and 
$\gcd(k,n)=1$.  We show parts (ii) and (iii) of Theorem 
\ref{T:cutting_sequences} together.  As in the irrational 
case, we consider two trajectories: 
\begin{itemize}
\item $\tau_{k/n}^-$, which has slope $m^-=\Delta_s^-(k/n)$ 
and starts at the top of edge $E$;
\item $\tau_{k/n}^+$, which has slope $m^+=\Delta_s^+(k/n)$ 
and starts at the bottom of edge $E$.
\end{itemize}
We again develop these trajectories in the plane 
$\mathbb{R}^2$ and show that they lie in the stacking diagram 
of $w=w(k/n)$.  Note that our conditions above imply that 
$\tau_{k/n}^-$ starts on the bottom of edge $A$ on $R_1$ and 
$\tau_{k/n}^+$ starts at the top of edge $A$ on $R_0$.  
Let $(x_{j,k/n},y_{j,k/n})$ be the coordinates of the upper 
right corner of the $j$th box in the stacking diagram of $w$.  
The heights of $\tau_{k/n}^-$ and $\tau_{k/n}^+$ at the 
$x$-value $x_{j,k/n}$ are described by 
\begin{equation}\label{y^-}
y_{j,k/n}^-=m^-(x_{j,k/n}-(1-s))+1
\end{equation}
and
\begin{equation}\label{y^+}
y_{j,k/n}^+=m^+(x_{j,k/n}-(1-s))+s,
\end{equation}
respectively.  We first show that $y_{j,k/n}^-=y_{j,k/n}^+=y_{j,k/n}$ for $j=k+n-1$, which implies $\tau_{k/n}^-$ lies above $\tau_{k/n}^+$ until these two trajectories intersect at the upper right corner of box $R_{k+n-1}$ in the stacking diagram of $w$. We then prove for all $j<k+n-1$ that $y_{j,k/n}^-<y_{j,k/n}$ whenever $w_{j+1}=A$ and $y_{j,k/n}^+>y_{j,k/n}$ whenever $w_{j+1}=B$.  Together, these will show that the linear trajectories represented by $\tau_{k/n}^-$ and $\tau_{k/n}^+$ start and end at a corner of $X_s$ and have the same cutting sequence $w(k/n)$.  That is, $\tau_{k/n}^-$ and $\tau_{k/n}^+$ represent saddle connections on the boundary of a cylinder whose closed trajectories have slopes $m\in\left(\Delta_s^-(k/n),\Delta_s^+(k/n)\right)$ and the same cutting sequences as that of a linear trajectory with slope $k/n$ starting at a corner of $T^2$.

We begin by showing $y_{j,k/n}^-=y_{j,k/n}^+$ for $j=k+n-1$.  That is,
$$\Delta_s^-(k/n)\left(x_{j,k/n}-(1-s)\right)+1=\Delta_s^+(k/n)\left(x_{j,k/n}-(1-s)\right)+s$$
by equations (\ref{y^-}) and (\ref{y^+}).  Rearranging terms, this is the same as showing that
\begin{equation}\label{y^+=y^-}
\left(\Delta_s^+(k/n)-\Delta_s^-(k/n)\right)\left(x_{j,k/n}-(1-s)\right)=1-s.\end{equation}
By Theorem \ref{T:delta}(vii), the difference between $\Delta_s^+(k/n)$ and $\Delta_s^-(k/n)$ is $s^{n-2}(1-s)^2/(1-s^n)$. Note that whenever $j=k+n-1$, 
\begin{equation}\label{x_{j,k/n}}
x_{j,k/n}=1+\sum_{i=1}^{\left\lfloor\frac{(k+n-1)n}{k+n}\right\rfloor}\left(\frac1s\right)^{i-1}=1+\sum_{i=1}^{n-1}\left(\frac1s\right)^{i-1}=1+\frac{\left(\frac1s\right)^{n-1}-1}{\frac1s-1}=1+\frac{s^{2-n}-s}{1-s}.
\end{equation}
With these two observations, the left-hand side of equation (\ref{y^+=y^-}) becomes
\begin{align*}
\frac{s^{n-2}(1-s)^2}{1-s^n}\left(\frac{s^{2-n}-s}{1-s}+s\right)&=(1-s)\frac{1-s^{n-1}+s^{n-1}(1-s)}{1-s^n}\\
&=(1-s)\frac{1-s^{n-1}+s^{n-1}-s^n}{1-s^n}=1-s
\end{align*}
so that $y_{j,k/n}^+=y_{j,k/n}^-$ for $j=k+n-1$.

Next, we show that $y^+_{j,k/n}=y_{j,k/n}$ for $j=k+n-1$.  Plugging in this value of $j$, we find 
\begin{equation}\label{y_{j,k/n}}
y_{j,k/n}=\left(\frac1s\right)^{\left\lfloor (k+n-1)n/k+n\right\rfloor}+\sum_{i=1}^{\left\lfloor(k+n)k/k+n\right\rfloor}\left(\frac1s\right)^{\left\lfloor (i-1)n/k\right\rfloor}=\left(\frac1s\right)^{n-1}+\sum_{i=1}^{k}\left(\frac1s\right)^{\left\lfloor (i-1)n/k\right\rfloor}.
\end{equation}
Equation (\ref{y^+}) and Theorem \ref{T:delta}(ix) tell us that
$$y^+_{j,k/n}=\left(\frac{1-s}{s}+\frac{s^{n-2}(1-s)}{1-s^n}\sum_{i=1}^{k}\left(\frac1s\right)^{\left\lfloor (i-1)n/k\right\rfloor}\right)(x_{j,k/n}-1+s)+s.$$
Using equation (\ref{x_{j,k/n}}), it becomes evident that 
\begin{equation}\label{x-1+s}
x_{j,k/n}-1+s=\frac{s^{2-n}-s}{1-s}+s=\frac{s^{2-n}-s+s(1-s)}{1-s}=\frac{s^{2-n}-s^2}{1-s}.
\end{equation}
Then $y^+_{j,k/n}$ becomes 
\begin{align*}
y^+_{j,k/n}&=\left(\frac{1-s}{s}+\frac{s^{n-2}(1-s)}{1-s^n}\sum_{i=1}^{k}\left(\frac1s\right)^{\left\lfloor (i-1)n/k\right\rfloor}\right)\left(\frac{s^{2-n}-s^2}{1-s}\right)+s\\
&=\left(\frac1s\right)^{n-1}-s+\sum_{i=1}^{k}\left(\frac1s\right)^{\left\lfloor (i-1)n/k\right\rfloor}+s=y_{j,k/n}.
\end{align*}

Lastly, we show for all $j<k+n-1$ that $y_{j,k/n}^-<y_{j,k/n}$ whenever $w_{j+1}=A$ and $y_{j,k/n}^+>y_{j,k/n}$ whenever $w_{j+1}=B$.  Consider a trajectory $\tau$ with slope $k/n$ that starts at a corner on $T^2$.  Certainly there exists a trajectory $\tau'$ with slope $\xi'>k/n$ for $\xi'\notin\mathbb{Q}$ starting at a corner on $T^2$ such that the first $k+n-1$ letters in the cutting sequences of both $\tau$ and $\tau'$ are identical.  Let $(x_{j,\xi'},y_{j,\xi'})$ be the coordinates of the upper right corner of the $j$th box in the stacking diagram of $w(\xi')$.  From the irrational case above, we know there exists a trajectory $\tau^-_{\xi'}$ starting at the top of edge $E$ on $X_s$ that lies below $y_{j,\xi'}$ whenever $w_{j+1}=A$.  But we chose $\xi'$ so that $y_{j,\xi'}=y_{j,k/n}$ for $j<k+n-1$.  Then $y^-_{j,\xi'}<y_{j,k/n}$ whenever $w_{j+1}=A$.  Since $\xi'>k/n$, we know that the slope of $\tau^-_{\xi'}$ is greater than that of $\tau^-_{k/n}$; that is, $\Delta_s^\pm(\xi')>\Delta_s^-(k/n)$.  This implies that $y^-_{j,k/n}<y^-_{j,\xi'}<y_{j,k/n}$ whenever $w_{j+1}=A$ for $j<k+n-1$ as claimed.  An analogous argument shows that $y^+_{j,k/n}>y_{j,k/n}$ whenever $w_{j+1}=B$ for $j<k+n-1$.  This proves parts (ii) and (iii) of Theorem \ref{T:cutting_sequences}.

To prove Theorem \ref{T:cutting_sequences}(iv), we consider the cylinder of closed trajectories corresponding to the saddle connections above.  Note that if not for our convention that the first letter in a cutting sequence is $B$, then the first and last letters of the cutting sequence of $\tau_{k/n}^-$ would be ambiguous since the trajectory begins and ends at the singularity of the surface.  If one were to consider trajectories starting on edge $E$ that lie strictly between $\tau_{k/n}^-$ and $\tau_{k/n}^+$ until intersecting the upper right most corner of $R_{k+n-1}$, it becomes clear that these are critical---rather than closed---trajectories.  Because $\tau_{k/n}^-$ has the same cutting sequence as a trajectory of slope $k/n$ on $T^2$, it crosses edge $A$ of $X_s$ $k$ times before returning to itself, with $k-1$ of these intersections occurring on the interior of edge $A$.  Since $\tau_{k/n}^+$ lies at the top of edge $A$, we wish to find the upper-most (interior) point on edge $A$ that $\tau_{k/n}^-$ crosses so that we can consider the corresponding cylinder of closed trajectories on $X_s$.  We claim that a trajectory of slope $\Delta_s^-(k/n)$ starting at a height of $\frac{s-s^n}{1-s^n}$ on edge $A$ of $X_s$ is the saddle connection described above.  It suffices to show that the line described by 
$$y^-=\Delta_s^-(k/n)\big(x-(1-s)\big)+\frac{s-s^n}{1-s^n}$$ 
intersects the right edge $A$ of $R_{k+n-1}$ at a height proportional to $\frac{s-s^n}{1-s^n}$ on the original edge $A$ of $R_0$ after scaling by $(1/s)^n$.  This will show that any other line with slope $\Delta_s^-(k/n)$ starting at height $\frac{s-s^n}{1-s^n}+\varepsilon$, for $\varepsilon\neq 0$, on edge $A$ of $R_0$ cannot represent the lower saddle connection of the cylinder starting at the uppermost point on $X_s$, for this line will intersect the right edge $A$ of $R_{k+n-1}$ at a point $\varepsilon$ higher or lower than did $y^-$ (the size of the right edge of $R_{k+n-1}$ has been scaled by at least $1/s$, so this line can't possibly be ``periodic'').  Now, a trajectory beginning at a height of $\frac{s-s^n}{1-s^n}$ starts at a distance of $s-\frac{s-s^n}{1-s^n}=\frac{s-s^{n+1}-s+s^n}{1-s^n}=s^n\frac{1-s}{1-s^n}$ from the top of edge $A$.  Assuming the trajectory crosses edge $A$ $n$ times, we scale this distance by $(1/s)^n$, so the distance from the upper right corner of $R_{k+n-1}$ to the trajectory on the right edge $A$ of said box should be $\frac{1-s}{1-s^n}$.  What we then aim to show is
$$\Delta_s^-(k/n)(x_{j,k/n}-(1-s))+\frac{s-s^n}{1-s^n}=y_{j,k/n}-\frac{1-s}{1-s^n}$$
for $j=k+n-1$.  Replacing $\Delta_s^-(k/n)$ with its value from Theorem \ref{T:delta}(ix) and using equation (\ref{x-1+s}), the left hand side becomes
\begin{align*}
&\Bigg(\frac{(s-s^n)(1-s)}{s^2(1-s^n)}+\frac{s^{n-2}(1-s)}{(1-s^n)}\sum_{i=1}^{k}\bigg(\frac1s \bigg)^{\left\lfloor(i-1) n/k \right\rfloor}\Bigg)\left(\frac{s^{2-n}-s^2}{1-s}\right)+\frac{s-s^n}{1-s^n}\\
=&\frac{(s-s^n)(s^{-n}-1)}{1-s^n}+\sum_{i=1}^{k}\bigg(\frac1s \bigg)^{\left\lfloor(i-1) n/k \right\rfloor}+\frac{s-s^n}{1-s^n}\\
=&\frac{s^{-n+1}-s-1+s^n+s-s^n}{1-s^n}+\sum_{i=1}^{k}\bigg(\frac1s \bigg)^{\left\lfloor(i-1) n/k \right\rfloor}\\
=&\bigg(\frac1s\bigg)^{n-1}+\sum_{i=1}^{k}\bigg(\frac1s\bigg)^{\left\lfloor (i-1)n/k\right\rfloor}-\frac{1-s}{1-s^n}.
\end{align*}
Equation (\ref{y_{j,k/n}}) tells us that this is the same as $y_{j,k/n}-\frac{1-s}{1-s^n}.$

Finally, we must show that a trajectory with slope 
$m \in (\Delta_s^-(k/n),\Delta_s^+(k/n))$ starting 
at a point on edge $A$ at height 
\begin{equation}\label{y_0}
y_0 
= \frac{s^2}{1-s}m 
  - \frac{s^n}{1-s^n} 
    \sum_{\ell=1}^{k} 
    \left(\frac1s\right)^{\left\lfloor\frac{(\ell-1)n}{k}\right\rfloor}
\end{equation}
is closed and has the same cutting sequence as a 
trajectory with slope $k/n$ on $T^2$.  Since we've 
already found the upper and lower saddle connections 
of this cylinder, we simply need to show that a trajectory 
with slope $m$ starting at $y_0$ on the left edge $A$ of $R_0$ will intersect 
the same point on the right edge $A$ of $R_{k+n-1}$.  
The line describing this trajectory is 
\[
y = m\big(x-(1-s)\big)+y_0.
\]
The trajectory's initial distance from the top of edge $A$ 
is $s-y_0$, and after the trajectory intersects $A$ $n$ times, 
this distance is scaled by $(1/s)^n$.  We then want the trajectory 
to be at a height of $y_{j,k/n}-(s-y_0)\big(\frac1s\big)^n$ whenever 
$x=x_{j,k/n}$ for $j=k+n-1$. In other words, we must show
\[
y_{j,k/n}-(s-y_0)\bigg(\frac1s\bigg)^n
= m\big(x_{j,k/n}-(1-s)\big)+y_0
\]
for $j=k+n-1$. Using equations \eqref{y_{j,k/n}} and \eqref{y_0}, 
the left hand side becomes
\begin{align*}
&\left(\frac1s\right)^{n-1}
+\sum_{i=1}^{k}
 \left(\frac1s\right)^{
 \left\lfloor (i-1)n/k \right\rfloor}
 - \left(s-\left(\frac{s^2}{1-s}m 
 - \frac{s^n}{1-s^n} 
    \sum_{\ell=1}^{k} 
    \left(\frac1s\right)^{\left\lfloor (\ell-1)n/k\right\rfloor}
 \right)\right)\left(\frac1s\right)^n \\
    =&\left(\frac1s\right)^n\frac{s^2}{1-s}m-\frac{s^n}{1-s^n}\sum_{i=1}^{k}\left(\frac1s\right)^{\left\lfloor (i-1)n/k \right\rfloor}\\
    =&\left(\left(\frac1s\right)^n-1\right)\frac{s^2}{1-s}m+y_0\\  
=&\left(\frac{s^{2-n}-s^2}{1-s}\right)m+y_0,
\end{align*}
which is $m(x_{j,k/n}-(1-s))+y_0$ by equation \eqref{x-1+s}.

\section{Classifying trajectories}\label{S:proofs}

In this section we draw together the tools developed in 
sections 2 and 3 in order to prove Theorems \ref{T:main1} 
and \ref{T:main2}. 

First, we define the sets $U_s$, $C_s$, and $C_s'$ whose 
properties are described in Theorem~\ref{T:main1}.
\begin{itemize}
\item As in \S\ref{S:staircase}, $C_s$ is the closure 
of the image of either $\Delta_s^-$ or $\Delta_s^+$ (or, 
equivalently, the union of their two images).
\item $U_s$ is the complement of $C_s$.
\item $C_s'$ is the set of slopes of the form 
$\Delta_s^+(k/n)$ or $\Delta_s^-(k/n)$ for some 
$k/n\in\Q$.
\end{itemize}
In other words, $C_s'$ is the set of endpoints of intervals 
in $U_s$. This relationship is parallel to the relationship 
between attracting cycles and saddle connections, as we 
shall see.

\subsection{Reversing directions}

As observed in \S\ref{SS:affine}, the affine group of 
$X_s$ contains an involution $\rho_s$ with derivative 
$-\mathrm{id}$. This map exchanges the subsurfaces 
$X_s^+$ and $X_s^-$, and it reverses the directions of 
linear trajectories, switching backward and forward. 
The material about forward trajectories in $X_s^+$ from 
\S\ref{S:symbolic} can therefore be transferred to 
backward trajectories in $X_s^-$. Notice also that for 
each slope, $X_s$ has six critical trajectories (because 
the cone angle at the singular point is $6\pi$), three 
of which are forward and three backward. (A saddle connection 
is both a forward and a backward critical trajectory.)

\subsection{Cylinder boundaries}

The set $C_s'$ of Theorem~\ref{T:main1}(iii) consists of 
the endpoints of the connected components of $U_s$. That 
is, $C_s'$ contains all slopes of the form $\Delta_s^+(k/n)$ 
or $\Delta_s^-(k/n)$. Suppose $m = \Delta_s^+(k/n)$. Then by 
Theorem~\ref{T:cutting_sequences}(ii), there is a saddle 
connection $\tau_m$ with slope $m$ that starts in the forward 
direction from the bottom of edge $E$ on the left side of 
$X_s^+$ and ends at the top of edge $A$ on the right side 
of $X_s^+$. It follows from Theorem~\ref{T:cantor} that all 
other forward trajectories in $X_s^+$ are asymptotic to 
$\tau_m$. A similar argument applies in the case where 
$m = \Delta_s^-(k/n)$.

\subsection{Attracting cycles}

If $m \in U_s$, then $m \in (\Delta_s^-(k/n),\Delta_s^+(k/n))$ 
for some $k/n \in \Q$. By Theorem~\ref{T:cutting_sequences}(iv), 
there exists a closed trajectory in the direction $m$ that 
crosses edge $A$ $n$ times and edge $B$ $k$ times. Its image 
$\Sigma_m^+$ is therefore an attracting cycle contained in 
an affine cylinder with scaling factor $s^n$. We now claim 
that every trajectory that is forward infinite, except the 
cycle $\Sigma_m^- = \rho_s(\Sigma_m)$, lies in the basin of 
attraction of $\Sigma_m^+$.

To show this claim, we adapt the proof of Theorem~\ref{T:cantor}. 
For these parameters $m$, there are two gaps in the image of 
$f : y \mapsto \{s(y+m)\}$, of lengths $t(1-s)$ and $(1-t)(1-s)$ 
for some $0 < t < 1$. Let $t$ be the value such that the gap of 
length $t(1-s)$ is the lower gap, which starts at $y = 0$. The 
intersections of a trajectory with edge $A$ follow an orbit 
of $f$. We thus consider the set 
$\bigcap_{N=1}^\infty \overline{f^N\big([0,1)\big)}$; any point 
$y_0$ in this set can be written 
\[
y_0 = \sum_{f^N\!(0)\le y_0} t(1 - s) s^{N-1}.
\]
The points of this form are precisely the intersections of 
the cycle $\Sigma_m^+$ with $A$. Therefore any other forward 
infinite trajectory in $X_s^+$ is eventually ``trapped'' in 
the affine cylinder of $\Sigma_m^+$.

Meanwhile, any backward infinite trajectory in $X_s^+$, 
besides $\tau$, eventually reaches edge $E$ and leaves $X_s^+$. 
Correspondingly, any forward infinite trajectory that starts 
in $X_s^-$ and is not contained in $\Sigma_s^-$ 
will eventually reach $E$ and lies in the basin 
of attraction for $\Sigma_s^+$.

\subsection{Attracting laminations} If 
$m \in C_s \setminus C_s'$, then $m$ has the form 
$\Delta_s^\pm(\xi)$ for a unique $\xi \notin \Q$.
In this case, let $\Sigma_m^+$ be the closure of either 
$\tau_\xi^+$ or $\tau_\xi^-$, as defined in 
\S\ref{irrational case} (their closures are the same). 
By Theorem~\ref{T:cantor}, the intersection of $\Sigma_m^+$ 
with the edge $A$ is the Cantor set $K_{s,\xi}$. By 
Theorem~\ref{T:hdim}, the Hausdorff dimension of 
$K_{s,\xi}$ is zero, and consequently $\Sigma_m^+$ is a 
minimal lamination. Let $\Sigma_m^- = \rho_s(\Sigma_m^+)$.

Theorem~\ref{T:cutting_sequences}(i) implies that the 
complement of $\Sigma_m^+$ in $X_s^+$ is a connected 
half-infinite open strip. Applying the involution $\rho_s$ 
shows that the complement of $\Sigma_m^+ \cup \Sigma_m^-$ 
in $X_s$ is an open strip, infinite in both directions. 
A trajectory moving in the forward direction within this 
strip eventually remains inside any neighborhood of 
$\Sigma_m^+$, and so $\Sigma_m^+$ is a forward attractor, 
and its basin of attraction is the complement of $\Sigma_m^-$.

\subsection{Cutting sequences}

We are now ready to prove Theorem~\ref{T:main2}. The 
direction $m \in \R$ has a closed trajectory if and only 
if $m \in \big(\Delta_s^-(k/n),\Delta_s^+(k/n)\big)$ 
for some $k/n \in \Q$, in which case by 
Theorem~\ref{T:cutting_sequences}(iv) $\tau$ has 
the same cutting sequence as a trajectory on $T^2$ 
with slope $k/n$. Any other forward infinite trajectory 
eventually falls in the corresponding affine cylinder, and 
thereafter has the same cutting sequence as $\tau$.

On the other hand, if $m \in C_s \setminus C_s'$, then any 
trajectory that is contained in $\Sigma_s^+$ is in the closure 
of a trajectory $\tau$ whose cutting sequence is Sturmian. 
Any trajectory that is not contained in either $\Sigma_s^+$ 
or $\Sigma_s^-$ is contained in the infinite strip of 
trajectories with the same cutting sequence as $\tau$. 

We conclude that every forward infinite trajectory on $X_s$ 
has a cutting sequence that is either eventually periodic 
or eventually Sturmian.

\section{Related surfaces}\label{S:related}

We conclude by transferring our results to two families 
of surfaces that are closely related to $X_s$. We use the 
observation that $X_s^+$ and $X_s^-$ can be independently 
deformed, and the understanding of trajectories from 
sections \ref{S:staircase} and \ref{S:symbolic} can be 
applied separately to these deformed subsurfaces, to give 
global information about trajectories.

\subsection{Twisting}\label{SS:twisting}

An affine deformation by the matrix 
$M^v = \left(\begin{smallmatrix}
1 & 0 \\ v & 1
\end{smallmatrix}\right)$ transforms trajectories with 
slope $m$ into trajectories with slope $m + v$. Let 
$X_s^u$ be the surface obtained from $X_s$ after deforming 
$X_s^-$ by $M^{u/s}$. (See Figure~\ref{F:twist}.) When $u$ 
is an integer, $M^{u/s}$ acts as a power of a full Dehn 
twist on $X_s^-$, and so $X_s^u$ is isomorphic to $X_s$. 

\begin{figure}[h]
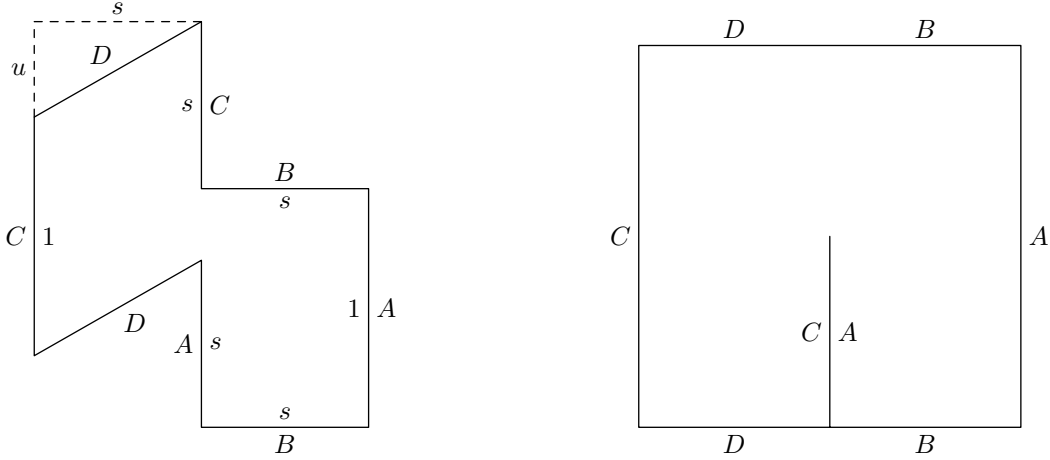

\includegraphics{surfacefig-6.mps}
\hspace{1in}
\includegraphics{surfacefig-7.mps}
\caption{{\sc Left:} A surface $X_s^u$ with $0 < u < 1$. 
Lower-case letters indicate dimensions. Capital letters 
indicate gluings between edges.
{\sc Right:} The surface $X_{1/2}^{1/2}$, which was the 
first surface the authors considered.}
\label{F:twist}
\end{figure}

For non-integer values of $u$, however, a trajectory on 
$X_s^u$ can have different forward and backward behavior. 
Briefly, given two slopes $m^+, m^-$, set $u = s(m^+ - m^-)$. 
Then $M^{u/s}$ transforms trajectories with slope $m^-$ 
into trajectories with slope $m^+$, and so a trajectory 
on $X_s^u$ having slope $m^+$ will behave like a trajectory 
on $X_s^+$ with slope $m^+$ when on the right half of $X_s^u$ 
and like a trajectory on $X_s^-$ with slope $m^-$ when on the 
left half of $X_s^u$. In particular, we can construct a 
surface $X_s^u$ for which there exists a direction $m$ 
exhibiting any of the following behaviors:
\begin{itemize}
\item $m$ has an attracting cycle with scaling factor $h^+$ 
and a repelling cycle with scaling factor $h^-$ for any choice 
of $h^+, h^- \in (0,1)$;
\item $m$ has an attracting lamination and a repelling cycle; 
\item $m$ has an attracting lamination and a repelling 
lamination, and these two laminations are not homeomorphic.
\end{itemize}

\subsection{Scaling}\label{SS:scaling}

We can use different parameters to determine the relative 
shapes of the two ``halves'' of the surface, as well. To 
be specific, given $s_1,s_2 \in (0,1)$, let $R_{s_1}^+$ be 
a rectangle with horizontal side lengths $s_1/(1 - s_1)$ 
and vertical side lengths $1/(1 - s_1)$, and let $R_{s_2}^-$ 
be a rectangle with horizontal side lengths $s_2/(1 - s_2)$ 
and vertical side lengths $1/(1 - s_2)$. (These rectangles 
are similar to the ones defined in \S\ref{SS:mainexample}; 
for purposes of constructing a homothety surface, only the 
similarity class of each polygon matters, which justifies 
our reuse of the same notation.) Join the top left edge of 
$R_{s_1}^+$ to the bottom right edge of $R_{s_2}^-$ along 
a segment of length $1$, and identify the remaining sides 
of $R_{s_1}^+$ and $R_{s_2}^-$ as shown on the right side 
of Figure~\ref{F:scaling}. Call the resulting surface 
$X_{s_1,s_2}$.

\begin{figure}[h]
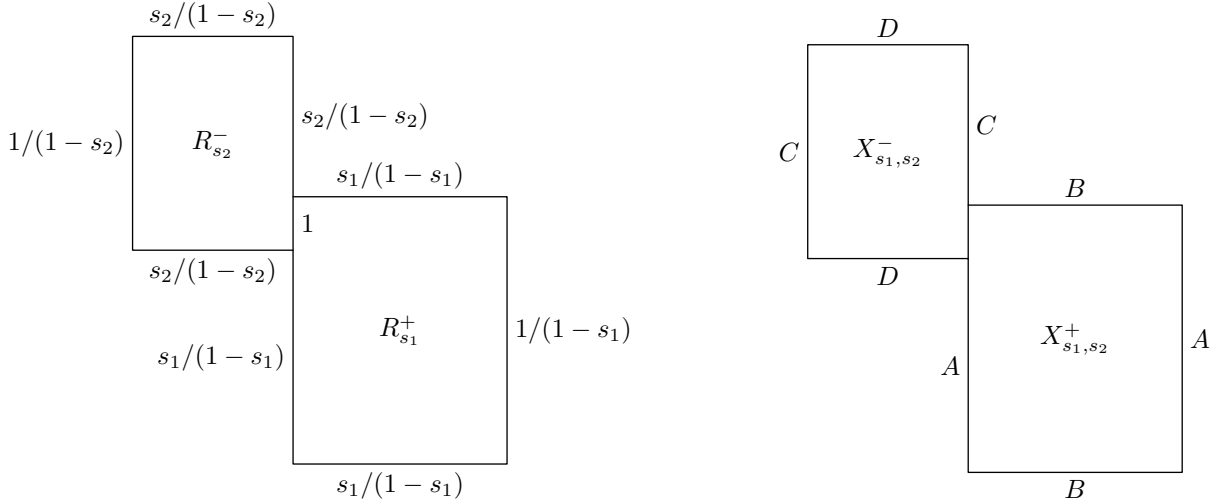

\includegraphics{surfacefig-8.mps}
\hspace{0.65in}
\includegraphics{surfacefig-9.mps}
\caption{{\sc Left:} The rectangles $R_{s_1}^+$ and $R_{s_2}^-$.
{\sc Right:} Edge identifications to form the surface 
$X_{s_1,s_2}$. $X_{s_1,s_2}^+$ is isomorphic to $X_{s_1}^+$, 
and $X_{s_1,s_2}^-$ is isomorphic to $X_{s_2}^-$, as defined 
in \S\ref{SS:mainexample}.}
\label{F:scaling}
\end{figure}

With this construction, we can again exhibit all of the 
trajectory behaviors mentioned in \S\ref{SS:twisting}. 
Indeed, given any $\xi > 0$, $m > 1$, it is possible to find 
$s \in (0,1)$ such that $m = \Delta_s^\pm(\xi)$ (when 
$\xi \notin \Q$) or $m \in [\Delta_s^-(\xi),\Delta_s^+(\xi)]$ 
(when $\xi \in \Q$). Thus one can simply choose two values 
$\xi_1, \xi_2 > 0$ that will produce the desired behaviors, 
and then find values of $s_1, s_2$ so that 
$\Delta_{s_1}^+(\xi_1) = \Delta_{s_2}^+(\xi_2)$.

\end{document}